%% file: example_paper.tex
\documentclass{article}
\usepackage[accepted]{icml2023}
\usepackage{algorithm,algorithmic}
\usepackage{microtype}
\usepackage{graphicx}
\usepackage{subfigure}
\usepackage{caption}
\usepackage{booktabs} 
\usepackage{hyperref}

\usepackage{amsmath}
\usepackage{amssymb}
\usepackage{mathtools}
\usepackage{amsthm}
\usepackage{multirow} 
\usepackage[capitalize,noabbrev]{cleveref}

\theoremstyle{plain}
\newtheorem{theorem}{Theorem}[section]
\newtheorem{proposition}[theorem]{Proposition}
\newtheorem{lemma}[theorem]{Lemma}

\theoremstyle{definition}
\newtheorem{definition}[theorem]{Definition}
\newtheorem{assumption}[theorem]{Assumption}
\theoremstyle{remark}
\newtheorem{remark}[theorem]{Remark}

\usepackage[textsize=tiny]{todonotes}
\icmltitlerunning{Accelerated Primal-Dual Methods for Convex-Strongly-Concave Saddle Point Problems}

\begin{document}
\include{macros}
\twocolumn[
\icmltitle{Accelerated Primal-Dual Methods for\\
				 Convex-Strongly-Concave Saddle Point Problems}



\icmlsetsymbol{equal}{*}

\begin{icmlauthorlist}
\icmlauthor{Mohammad Khalafi}{yyy}
\icmlauthor{Digvijay Boob}{yyy}
\end{icmlauthorlist}

\icmlaffiliation{yyy}{Department of Operations Research and Engineering Management, Southern Methodist University, Dallas TX, USA}

\icmlcorrespondingauthor{Mohammad Khalafi}{mohamadk@smu.edu}

\icmlkeywords{Machine Learning, ICML}

\vskip 0.3in
]



\printAffiliationsAndNotice{\icmlEqualContribution} 
\input{sec-abstract}

\input{sec-intro}

\section{Notation and Definitions} \label{sec:notations and defnitions}
We use  $\|\cdot\|_q$ and $\|\cdot\|$ to denote $\ell_{q}$-norm  and Euclidean norm of any vector, respectively. $\langle\cdot,\cdot\rangle$ stands for the standard inner product of two vectors. For a general function $h$, $\nabla h$ expresses the gradient of $h$. $\nabla_v h$ implies the partial gradient of $h$ with respect to variable $v$. 
We use $[m]$ to denote $\{1, \dots, m\}$.
For a compact set $\mathcal{W}$, we define its diameter $D_{\mathcal{W}} := \max_{w^\prime,w\in \mathcal{W}} \|w^\prime-w\|/\sqrt{2}$. We use $z =(x,y)$ as the combined variable defined on the set $X\times Y \equiv Z$. We naturally extend this notation for $\bar{z} =(\bar{x}, \bar{y})$, $z_t =(x_t, y_t), \
\bar{z}_t =(\bar{x}_t, \bar{y}_t) $ and so on.

{\bf Problems setting.} 
In problem \eqref{eq : SPP definition}, $X \subseteq \mathbb{R}^n$ and $Y \subseteq \mathbb{R}^m$ are compact convex sets, $f:X\rightarrow \mathbb{R}$ is a convex primal function,  $g:Y\rightarrow\mathbb{R}$ is a convex dual function and $\phi(x,y):X\times Y\rightarrow \mathbb{R}$ is a convex-concave coupling function, i.e., $\phi(\cdot, y)$ is convex for all $y \in Y$ and $\phi(x, \cdot)$ is concave for all $x \in X$.	
The \textit{gap function} defined below acts as a measure of convergence.
\begin{definition}\label{def: gap function}
	For a point  $\bar{z} \in Z$, we define its gap as
	\begin{equation*}
		\gap(\bar{z}) = \max_{z\in Z} Q(\bar{z},z).
	\end{equation*}
where $Q(\bar{z},z):= \mathcal{L}(\bar{x},y) - \mathcal{L}(x,\bar{y}).$
\end{definition}	
It is easy to see that $\gap(\bar{z}) \ge 0$ and ${z^\star} \in Z$ is the saddle point for \eqref{eq : SPP definition} if and only if $\gap(z^\star)=0$. Hence, we can measure the quality of an approximate solution using the $\gap$ function. 
\begin{definition}
	For $\epsilon > 0$, we say that $\bar{z} \in Z$ is an $\epsilon$-solution of problem \eqref{eq : SPP definition} if $\gap(\bar{z})\le \epsilon$.
\end{definition}
%
%
	We call a  function $h:H\rightarrow \mathbb{R}$ to be strongly-convex with modulus $\mu_h>0$ if it satisfies
	$h(x')-h(x) - \langle\,\nabla h(x) ,x'-x\rangle\geq \tfrac{\mu_h}{2}\|x'-x\|^2$ for all $x',x\in H$


Throughout the paper, we make the following assumptions on the general coupling function $\phi(x,y)$:
\begin{assumption}
	We assume function $\phi(\cdot,y)$ is $L_{xx}$-smooth for all $y \in Y$, $\phi(x, \cdot)$ is $L_{yy}$-smooth for all $x \in X$ and $\phi$ is $L_{xy}$-smooth, i.e., $\phi$ satisfies the following relations, respectively, for all $x, x' \in X, \ y,y' \in Y$:
	\begin{align*}
		\|\nabla_x\phi(x',y) - \nabla_x\phi(x,y)\| &\leq L_{xx}\|x'-x\|, 
		\\
		\|\nabla_y\phi(x,y') - \nabla_y\phi(x,y)\|&\leq L_{yy}\|y'-y\|, 
		\\
		\|\nabla_y\phi(x',y) - \nabla_y\phi(x,y)\|&\leq L_{xy}\|x'-x\| .
	\end{align*}
\end{assumption}

%
%
If all Lipschitz constants above are positive, then $\phi(x,y)$ is a general nonlinear coupling function. If either $L_{xx}=0$ or $L_{yy}=0$, then the coupling function is linear in $x$ or $y$, respectively. We refer to these cases as the {\em semi-linear coupling}. $L_{xx} = L_{yy} = 0$ implies a bilinear coupling. 
\vspace{-5pt}
\input{sec-tech-overview}

\section{The ALPD method for general $\phi$}\label{sec:APLD}
In addition to the primal acceleration mentioned in earlier section, we consider two more generalizations: (i) we use the linear approximation for $g$ instead of its proximal operator to allow the use of complex dual functions, (ii) the coupling function $\phi$ is a general nonlinear function.

To address the issues mentioned in Section \ref{sec:Technical Overview} in the broader settings above, we present the accelerated linearized primal-dual (ALPD) method (see Algorithm \ref{alg: Accelerated Lin PD for SPP linear g}). Here, we introduce a new parameter $\beta$, which is motivated from a (three-sequence) form of Nesterov's AGD algorithm \cite{nesterov1983method}. If we set $\beta_{t} = 1$ in Algorithm \ref{alg: Accelerated Lin PD for SPP linear g}, then it is easy to see that $\underline{x}_t = x_t$ and $\bar{x}_{t+1} = x_{t+1}$ for all $t$, and we immediately recover the LPD method for the bilinear coupling $\phi(x,y) = y^\top Ax$. Hence, the ALPD method is a generalization of the LPD method in two senses: (i) using the parameter $\beta_{t} \ge 1$, we aim to put the AGD framework inside the LPD and reduce the impact of $L_f$ in the complexity, and (ii) using a new sequence $\{v_t\}$ in place of $\{A\tilde{x}_t\}$, we allow for the nonlinear coupling function $\phi$. 


\begin{algorithm}[t]
	\caption{ Accelerated Linearized PD (ALPD) method  }\label{alg: Accelerated Lin PD for SPP linear g}
	\begin{algorithmic}[1]
		\STATE {\bf Initialize} $\bar{x}_1=x_0=x_1\in X , \bar{y}_1=y_0=y_1\in Y$
		\FOR{$t = 1, \ldots, K$}
		\STATE $\underline{x}_t \gets (1-\beta^{-1}_t)\bar{x}_t + \beta^{-1}_t x_t $
		\STATE$v_t \gets (1+\theta_t)\nabla_y\phi(x_t,y_t)-\theta_{t}\nabla_y\phi(x_{t-1},y_{t-1})$
		\STATE $y_{t+1}\gets \arg\min\limits_{y \in Y} \langle\,-v_t + \grad g(y_t),y\rangle +\tfrac{1}{2\tau_{t}}\|y-y_t\|^2$
		\STATE $x_{t+1}\gets \arg\min\limits_{x\in X} \langle\,\grad f(\underline{x}_t)+\grad_x\phi(x_t,y_{t+1}),x \rangle +\tfrac{1}{2\eta_{t}}\|x-x_t\|^2$
		\STATE $\bar{x}_{t+1} = (1-\beta^{-1}_t)\bar{x}_t+ \beta^{-1}_t x_{t+1}$
		\STATE $\bar{y}_{t+1} = (1-\beta^{-1}_t)\bar{y}_t+ \beta^{-1}_t y_{t+1}$
		\ENDFOR
		\STATE {\bf return} $\bar{x}_{K+1},\bar{y}_{K+1} $
	\end{algorithmic}
\end{algorithm}

The following lemma provides a useful recursive relation on the primal-dual gap function of the iterates of Algorithm \ref{alg: Accelerated Lin PD for SPP linear g}. It is later used for bounding the gap function (see Definition \ref{def: gap function}). See Appendix \ref{apx:ALPD} for proofs of all results in this section.
\begin{lemma}\label{lemma: lamma2}
	Let $\bar{z}_{t+1} = (\bar{x}_{t+1}, \bar{y}_{t+1})$ then:
	\vspace{-3pt}
	\begingroup
	\allowdisplaybreaks
	\begin{align}
			&\beta_t Q(\bar{z}_{t+1},z)- (\beta_t-1)Q(\bar{z}_{t},z) \nonumber\\
			&\leq \tfrac{1}{2\eta_t}\big[\|x- x_t\|^2 -\|x- x_{t+1}\|^2\big] \nonumber\\
			&\quad+\big[ \big(\tfrac{1}{2\tau_t}-\tfrac{\mu_g}{2}\big)\|y- y_t\|^2
			 - \tfrac{1}{2\tau_t}\|y_- y_{t+1}\|^2 \big] \nonumber\\
			&\quad -\big(\tfrac{1}{2\eta_t} -\tfrac{L_{{f}}}{2\beta_t} - \tfrac{L_{xx}}{2} \big)\|x_t- x_{t+1}\|^2\nonumber\\
			&\quad-\big(\tfrac{1}{2\tau_t}- \tfrac{L_{{g}}}{2}\big)\tfrac{1}{2}\|y_t- y_{t+1}\|^2\nonumber\\
			&\quad + [\phi(x_{t+1},y)- \phi(x_{t+1},y_{t+1}) -\langle\,v_t,y-y_{t+1} \rangle]\label{eq: lemma 3.1}			
	\end{align}
\endgroup
\end{lemma}
\vspace{-10pt}
Lemma \ref{lemma:lemma3} states a step-size condition for parameters $\{\beta_{t}, \theta_{t}, \gamma_{t}, \tau_{t}, \eta_t\}$ and provides an upper bound on the $\gap(\bar{z}_{K+1})$ where $\bar{z}_{K+1}$ is the output of the ALPD method.
\begin{lemma}\label{lemma:lemma3}
	Suppose $\{\beta_{t}, \theta_{t}, \gamma_{t}, \tau_{t}, \eta_t\}$ satisfy
	\vspace{-5pt}
	\begin{align}
			&\beta_1 = 1, \quad &&\beta_{t+1} -1 = \beta_t \theta_{t+1},\nonumber\\
			&\theta_{t} = \tfrac{\gamma_{t-1}}{\gamma_{t}}, &&0\leq \theta_{t} \leq \tfrac{\tau_{t-1}}{\tau_{t}}, \nonumber\\
			& \tfrac{1}{2\eta_{t}}-\tfrac{L_{{f}}}{2\beta_{t}}-2L_{xy}^2\tau_{t}\geq \tfrac{L_{xx}}{2},\nonumber\\
			&\tfrac{1}{4\tau_{t}}-\tfrac{L_{{g}}}{2}- 2L_{yy}^2\tau_{t}\geq 0, \label{con: stepsize condition}
	\end{align}
	\vspace{-5pt}
	then, we have
	\begin{align}
			&\beta_K\gamma_{K}Q(\bar{z}_{K+1},z)
			\leq B_K(z,z_{[K]}) \nonumber\\
			&\quad	+ \gamma_{K}\langle\,\nabla_y\phi(x_{K+1},y_{K+1})-\nabla_y\phi(x_{K},y_{K}),y-y_{K+1}\rangle \nonumber\\
			&\quad-\gamma_{K}\left(  \tfrac{1}{2\eta_{K}}-\tfrac{L_{{f}}}{2\beta_K}-\tfrac{L_{xx}}{2}\right)  \|x_{K+1}-x_K\|^2 \nonumber\\
			&\quad-\gamma_{K}\left(  \tfrac{1}{4\tau_K}-\tfrac{L_{{g}}}{2}\right) \|y_t- y_{K+1}\|^2,\label{eq: lemma 3.2}
	\end{align}
	where 
	\vspace{-5pt}
	\begin{align*}
			&B_K(z,z_{[K]}) := \sum_{t=1}^K \{\tfrac{\gamma_t}{2\eta_t}[\|x-x_t\|^2-\|x-x_{t+1}\|^2]\\
			&\quad\quad\quad\quad+\gamma_t\big( \tfrac{1}{2\tau_{t}}-\tfrac{\mu_g}{2}\big)\|y-y_{t}\|^2-\tfrac{\gamma_t}{2\tau_t}\|y-y_{t+1}\|^2\}.
	\end{align*}
\end{lemma}
\vspace{-5pt}
A comparison of the ALPD step-size conditions in \eqref{con: stepsize condition} with the LPD (in \eqref{eq : conditions PD}) shows that the impact of $L_f$ can be mitigated using the parameter $\beta_t$. Indeed, for the bilinear problems, i.e., $L_{xx} =0$ and $L_{xy} = \|A\|$, the fifth relation in \eqref{con: stepsize condition} reveals the necessity of condition $\tfrac{1}{\eta_t} \ge \tfrac{L_f}{\beta_t}$ for the ALPD method. Appropriate choice of $\beta_t$, (say, increasing with $t$)  may allow us to increase  $\eta_t$ resulting in a stronger learning rate. Besides, \eqref{eq:conditionPD4} requires $\tfrac{1}{\eta_t} \ge L_f$ and hence, no scope for improving the learning rate. 
Theorem \ref{Thm:thm2} exhibits a tangible upper bound on the $\gap$ function that explicitly shows the dependence of the convergence rate on $\beta_t$. 
\begin{theorem}\label{Thm:thm2}
	In addition to the assumptions in Lemma \ref{lemma:lemma3}, let the following condition hold for $t\geq 2$:
	\vspace{-5pt} 
	\begin{equation}\label{eq: condition 5}
		\gamma_{t}(\tfrac{1}{\tau_{t}}-\mu_g)\leq 	\tfrac{\gamma_{t-1}}{\tau_{t-1}}, \quad \tfrac{\gamma_{t}}{\eta_{t}}\leq\tfrac{\gamma_{t-1}}{\eta_{t-1}} + \tfrac{L_{xx}}{2}
	\end{equation}
	Then, we have
	\begin{equation}\label{eq: Thm 1}
		\gap(\bar{z}_{K+1}) \leq \big(\tfrac{\gamma_{1}}{\beta_K\gamma_{K}\eta_1} +\tfrac{KL_{xx}}{\beta_{K}\gamma_{K}}\big)D_X^2+\tfrac{\gamma_{1}}{\beta_K\gamma_{K}\tau_1}D_Y^2.
	\end{equation}
	where $D_X^2 $ and $D_Y^2$ are diameters of set $X$ and $Y$.
\end{theorem}
\subsection{Step-size policy for the ALPD method}
Using the result of Theorem \ref{Thm:thm2}, we are ready to present step-size policy for the ALPD method. We break our analysis in two cases.
\subsubsection{Case 1: Semi-linear coupling with $L_{xx} = 0$}\label{sec: Case 1: Semi-linear coupling with $L_{xx} = 0$}
Assume a semi-linear coupling function $\phi(x,y)$ which is linear in $x$, i.e., $L_{xx} = 0$. Then, let us consider the following choice of parameters for Algorithm \ref{alg: Accelerated Lin PD for SPP linear g}:
\vspace{-5pt}
\begin{equation}\label{stepsize: case1}
	\setlength{\jot}{-1pt}
	\begin{aligned}
		\gamma_1 &= 1, \quad  \gamma_{t} = \tfrac{t+1}{2} + \tfrac{2\sqrt{2}L_{yy}+2L_{{g}}}{\mu_g}, \quad t \ge 2,\\
		\theta_{t} &= \tfrac{\gamma_{t-1}}{\gamma_{t}},\quad t\ge 2\\
		\beta_1 &= 1, \quad \beta_{t+1} = 1 + \theta_{t+1} \beta_{t}, \\
		\eta_t &= \tfrac{t+1}{5L_{{f}}+16 L_{xy}^2/\mu_g },\\
		\tfrac{1}{\tau_t} &=\tfrac{\mu_gt}{2} + 2\sqrt{2}L_{yy}+2L_{g}.
	\end{aligned}
\end{equation}
Comparing the step-size policy in \eqref{stepsize: case1} with conditions in \eqref{con: stepsize condition} where $L_{xx} = 0$, it is easy to see that the relations $\theta_{t} = \tfrac{\gamma_{t-1}}{\gamma_{t}}$, the recursive relation on $\beta_{t}$ and $\tfrac{1}{4\tau_{t}}-\tfrac{L_{{g}}}{2}- 2L_{yy}^2\tau_{t}\geq 0$ are satisfied. Furthermore, since $\gamma_t$ is increasing and $\tau_t$ is decreasing, we have $\theta_{t} < 1 < \tfrac{\tau_{t-1}}{\tau_t}$. It is straight-forward to see that $\{\tfrac{\gamma_t}{\eta_t}\}$ is a decreasing sequence. 
Besides,  by choosing $\tau_{t}$ according to this step-size policy, first condition in \eqref{eq: condition 5} also holds.
The proposition below provides a bound on $\beta_{t}$. 
\begin{proposition}\label{prop:prop1}
	Suppose we set the step-size parameters according to \eqref{stepsize: case1} then $\beta_{t+1} \in [\tfrac{t+2}{2},t+1]$.
\end{proposition}
Using the above proposition, we verify the one remaining condition of \eqref{con: stepsize condition} with $L_{xx} = 0$:
\begin{align*}
		&\tfrac{1}{2\eta_{t}}-\tfrac{L_{{f}}}{2\beta_{t}}-2L_{xy}^2\tau_{t} \\
		&\ge \tfrac{5L_{{f}}}{2(t+1)}- \tfrac{L_{{f}}}{2\beta_{t}}+ \tfrac{ 16L_{xy}^2}{2\mu_g(t+1)} - \tfrac{4L_{xy}^2}{\mu_gt+4\sqrt{2}L_{yy}}\geq 0,
\end{align*}
where the first inequality follows by replacing the values of $\eta_t, \tau_t$ along with the fact that $L_g \ge 0$, and the second inequality holds
since $\beta_{t}\geq \tfrac{t}{2}\geq \tfrac{t+1}{5}$ and $\tfrac{16L_{xy}^2}{2\mu_g(t+1)} - \tfrac{4L_{xy}^2}{\mu_gt}\geq 0 $ for $t \ge 1$.

Using \eqref{stepsize: case1} in Theorem \ref{Thm:thm2}, we obtain the following upper bound on the Gap: 
\begin{equation}\label{eq: upper bound semilinear}
	\gap(\bar{z}_{K+1})  \leq \tfrac{{1}}{\beta_K\gamma_{K}\eta_1} D_X^2+\tfrac{{1}}{\beta_K\gamma_{K}\tau_1}D_Y^2.
\end{equation}
Note that since $\beta_{t}$ and $\gamma_{t}$ are increasing at a linear rate, we obtain the accelerated convergence rate of $O(\tfrac{L_f + L_{yy}}{K^2} + \tfrac{L_{xy}^2}{\mu_gK^2})$ which is equivalent to the complexity of $K = O(\sqrt{\tfrac{L_f + L_{yy}}{\epsilon}} + \tfrac{L_{xy}}{\sqrt{\mu_g\epsilon}})$ for getting an $\epsilon$-solution of \eqref{eq : SPP definition}.
\begin{remark}\label{rem:compare_hamedani}
	\cite{hamedani2021primal} is the only known single-loop PD algorithm that shows accelerated convergence when the coupling function is semi-linear with $L_{yy} =0$ and $\mu_f > 0$. We have a (reflected) result where $\mu_g > 0$ and $L_{xx} =0$. Even then, \cite{hamedani2021primal} assume  $f$ and $g$ have proximal updates. Hence, they do not need any additional acceleration of the ALPD method.  
\end{remark}
\subsubsection{Case 2: nonlinear coupling}\label{sec:nonlinear_stepsize}
Now, let us consider an SPP with a general nonlinear coupling function, i.e., $L_{xx}>0$. In this case, using similar arguments as in Case 1, it is easy to see that the step-size policy in \eqref{stepsize: case1} with the following single change in $\eta_{t}$
\[
\eta_t = \tfrac{t+1}{5L_{{f}}+16 L_{xy}^2/\mu_g +(t+1)L_{xx}},
\]
satisfies the condition \eqref{con: stepsize condition} in Lemma \ref{lemma:lemma3}. Furthermore, \eqref{eq: condition 5} is also satisfied.
Thus, using Theorem \ref{Thm:thm2}, we can establish the following upper bound on the $ \gap $ function: \vspace{-4pt}
\begin{align*}
	\gap(\bar{z}_{K+1})  \leq \big(\tfrac{\gamma_{1}}{\beta_K\gamma_{K}\eta_1} +\tfrac{KL_{xx}}{\beta_{K}\gamma_{K}}\big)D_X^2+\tfrac{\gamma_{1}}{\beta_K\gamma_{K}\tau_1}D_Y^2.
\end{align*}
Though we get acceleration in terms of $L_f$, convergence rate in terms of  $L_{xx}$ is  of $\Ocal(\tfrac{1}{K})$. This is similar to the LPD case where the complexity had a weaker dependence on $L_f$. ALPD method does accelerated on the primal only term, i.e., $L_f$. However, accelerating the convergence for the primal coupling term is still difficult. In light of Remark \ref{rem:compare_hamedani}, accelerating the class of PD methods for the nonlinear coupling is a challenging open problem, even without linearization.

\section{The Inexact ALPD method for general $\phi$} 
\label{sec:Prox-APD in SPP with linearized functions in primal and dual problems}
This section proposes an Inexact ALPD method to improve the complexity in $L_{xx}$. The linearization of $\phi(x,y_{t+1})$ in the ALPD method generates errors that depend on $L_{xx}$. It leads to a slow convergence rate when $L_{xx} > 0$. To fix this issue, we use $\phi(x,y_{t+1})$ instead of its linearization in the $x$-update (compare line 6 of Algorithms \ref{alg: Accelerated Lin PD for SPP linear g} and \ref{alg: Accelerated Lin PD for SPP linear g prox}). However, we cannot evaluate the proximal oracle of $\phi(\cdot, y)$ efficiently. To evaluate the truly representative computational effort for this algorithm, we propose an inexact approach in the $x$-update and perform a detailed analysis of the inner loop to estimate the complexity bounds. The rest of this section is dedicated to the complexity analysis of Algorithm \ref{alg: Accelerated Lin PD for SPP linear g prox} in the outer loop and inner loop.
\begin{algorithm}[t]
	\caption{Inexact ALPD Method }\label{alg: Accelerated Lin PD for SPP linear g prox}
	\begin{algorithmic}[1]
		\STATE {\bf Initialize} $\bar{x}_1=x_0=x_1\in X ,\bar{y}_1=y_0= y_1\in Y$
		\FOR{$t = 1, \ldots, K$}
		\STATE $\underline{x}_t\gets (1-\beta^{-1}_t)\bar{x}_t + \beta^{-1}_t x_t $
		\STATE $v_{t}\gets(1+\theta_{t})\nabla_y\phi(x_t,y_t)- \theta_{t}\nabla_y\phi(x_{t-1},y_{t-1})$
		\STATE $y_{t+1}\gets\arg\min\limits_{y \in Y} \langle\,-v_t +\nabla g(y_t),y\rangle +\tfrac{1}{2\tau_{t}}\|y-y_t\|^2$
		\STATE $x_{t+1}$ is a $\delta_t$-approximate solution of the problem:
		\begin{equation}\label{eq:subprob_inexactALPD}
			\min_{x \in X} \langle\,\nabla f(\underline{x}_t),x \rangle + \phi(x,y_{t+1}) +\tfrac{1}{2\eta_{t}}\|x-x_t\|^2
		\end{equation}\vspace{-3mm}
		\STATE $\bar{x}_{t+1} \gets (1-\beta^{-1}_t)\bar{x}_t+ \beta^{-1}_t x_{t+1}$
		\STATE $\bar{y}_{t+1} \gets (1-\beta^{-1}_t)\bar{y}_t+ \beta^{-1}_t y_{t+1}$
		\ENDFOR
		\STATE  {\bf return} $\bar{x}_{K+1} ,\bar{y}_{K+1}$
	\end{algorithmic}
\end{algorithm}
\subsection{Complexity analysis of the Inexact ALPD method}\label{sec: step size policy AGD}
\textbf{Complexity analysis of the outer loop}\\
Using a proximal oracle of $\phi(\cdot, y_{t+1})$ in the $x$-update removes the linearization errors that depend on $L_{xx}$. Hence, the outer loop analysis reduces to Case \ref{sec: Case 1: Semi-linear coupling with $L_{xx} = 0$}. 
Consequently, applying conditions in (\eqref{con: stepsize condition}, \eqref{eq: condition 5}) and using the step-size policy (\eqref{stepsize: case1}), we have the following theorem. 

\begin{theorem}
	\begin{equation} \label{eq: step size prox}
		\begin{aligned}
			\gap(\bar{z}_{K+1})  \leq& \tfrac{\gamma_{1}}{\beta_K\gamma_{K}\eta_1} D_X^2+\tfrac{\gamma_{1}}{\beta_K\gamma_{K}\tau_1}D_Y^2\\
			&+\tfrac{ \sum_{t=1}^K \gamma_{t}\delta_t}{\beta_K\gamma_{K}}+\tfrac{\sum_{t=1}^K\gamma_t \sqrt{4\tfrac{1}{\eta_t}\delta_t}D_X^2}{\beta_K\gamma_{K}}.
		\end{aligned}
	\end{equation}
\end{theorem}
	The above upper bound is similar to \eqref{eq: upper bound semilinear}with the addition of the last two terms since we are using a $\delta_t$-approximate solution for \eqref{eq:subprob_inexactALPD}. The detailed proof and analysis of the inexact ALPD method is in Appendix \ref{apx:inexact-ALPD}. To manage the error caused by $\delta_t$, we require
	\vspace{-5pt}
	\begin{equation}\label{eq: delta}
		\delta_t = \tfrac{1}{t^c}.
	\end{equation}
	Note that we can choose $c = 3.5$ such that 
	$\tsum_{t=1}^K \gamma_t \delta_t$ and $\tsum_{t=1}^K \gamma_t \sqrt{\delta_t/\eta_t}$ are bounded by a constant.   
	Once the last two terms are bounded, we need $\mathcal{O}(\sqrt{\tfrac{1}{\epsilon}})$ evaluations in $\nabla f$, $\nabla g$ and $\nabla_y \phi$ to obtain an $\epsilon$-solution of \eqref{eq : SPP definition}. To compute the gradient complexity of $\grad_x\phi$ and the impact of the above choice of $c$, we perform the inner loop analysis below.
	
	\textbf{Complexity analysis of the inner loop}
	
	We implement the AGD method (\cite{nesterov2003introductory}) to solve the subproblem \eqref{eq:subprob_inexactALPD}. Let $k_t$ denote the number of AGD iterations for the $t$'th iteration in the outer loop . Consequently, the complexity of $\grad_x \phi$ after $K$ outer loop iterations is $\sum_{t=1}^K k_t$. The number of AGD iterations $k_t$ is directly related to the choice of error $\delta_t$.
 	\citet{nesterov2003introductory} shows that for a $L$-smooth and $\mu$-strongly convex function, we need $\mathcal{O}(\sqrt{\tfrac{L}{\mu}}\log(\tfrac{1}{\epsilon}))$ AGD iterations to obtain an $\epsilon$ error on the optimality. Then, to obtain a $\delta_t$ error on the optimality of \eqref{eq:subprob_inexactALPD}, we need $k_t = \mathcal{O}(\sqrt{L_{xx}\eta_t}\log(\tfrac{1}{\delta_t}))$  iterations of the AGD method. Here, we used $L=L_{xx}$ and $\mu = \tfrac{1}{\eta_t}$. Setting $\delta_t$ as per \eqref{eq: delta}, we obtain  $k_t = \mathcal{O}(\sqrt{L_{xx}\eta_t}\log(t^c))$ or $\Ocal(c\sqrt{L_{xx}\eta_t}\log(t))$. Hence, the total number of iterations of AGD for  $K$ outer iterations (equivalently, the number of gradients evaluations of $\grad_x \phi$) is 
	\begin{align*}
			\tsum_{t=1}^K k_t &=  \tsum_{t=1}^K c\sqrt{L_{xx}\eta_t}\log(t)\\
			& =  \tsum_{t=1}^K c\sqrt{L_{xx}\tfrac{t+1}{5L_{f}+16 L_{xy}^2/\mu_g }}\log(t)\\
			&\le c \sqrt{\tfrac{L_{xx}}{5L_{f}+16 L_{xy}^2/\mu_g }} (K+1)^{3/2} \log(K+1).
	\end{align*}
Using $K = \mathcal{O}(\sqrt{\tfrac{L_f+L_{xy}^2/\mu_g}{\epsilon}})$ in the above relation, and assuming $L_{xx}, L_f, L_{xy}^2/\mu_g$ are of $\mathcal{O}(L)$ we obtain
	\begin{align}
			\tsum_{t=1}^K k_t &=\mathcal{O} \big(c \sqrt{L_{xx}\sqrt{L_{f}+L_{xy}^2/\mu_g }} \tfrac{1}{\epsilon^{3/4}} \log(\tfrac{1}{\epsilon}) \big)\nonumber\\
			&= \mathcal{O}(\tfrac{L}{\epsilon^{3/4}} \log(\tfrac{1}{\epsilon})).\label{eq: complex analysis}
	\end{align}
	From the complexity analysis of the outer and inner loops, we can develop the following theorem 
	\begin{theorem}
		For obtaining an $\epsilon$-error in inexact ALPD, one needs $\mathcal{O}(\sqrt{\tfrac{L_f}{\epsilon}})$ and $\mathcal{O}(\sqrt{\tfrac{L_{yy}}{\epsilon}})$ evaluations in terms of $\grad f$ and $\grad_y \phi$ respectively. Moreover, suppose $L_{xx}, L_f$, and $ L_{xy}^2/\mu_g$ are of $\mathcal{O}(L)$, then we need $\tilde{\mathcal{O}}(\tfrac{L}{\epsilon^{3/4}}) $ evaluations of $\grad_x \phi$.  
	\end{theorem}
	Observe the impact of $c$ on the complexity of $\grad_x \phi$ is only of a constant factor. This happens since \eqref{eq:subprob_inexactALPD} is a strongly convex problem.  
	In view of \eqref{eq: complex analysis} and $K = \mathcal{O}(\tfrac{1}{\sqrt{\epsilon}})$, we show that the Inexact ALPD method exhibits the gradient complexity of $\mathcal{O}(\tfrac{1}{\sqrt{\epsilon}})$ for $\grad f, \grad g$ and $\grad_y \phi$. Moreover, its gradient complexity for $\grad_x \phi$ is of $\tilde{\mathcal{O}}(\tfrac{1}{\epsilon^{3/4}})$. In comparison, the ALPD method has $\mathcal{O}(\tfrac{1}{\epsilon})$ gradient complexity for $ \grad_x \phi$. The Inexact ALPD method improves the complexity in $L_{xx}$, and obtains optimal complexity in $L_f$.
	\input{sec-numerical-ex}
	\section{Conclusion}
	We showed that the standard LPD methods do not mitigate the impact of the linearization of the primal function for convex-strongly-concave SPP. Therefore, we designed the ALPD method which exhibits the optimal complexity for the semi-linear coupling case. 
	For the general nonlinear coupling, we designed a two-loop Inexact ALPD method that maintains the optimal gradient complexity of the primal function and significantly improves the gradient complexity of the coupling function. 
	We verified our findings through numerical experiments.
	\bibliography{reference}
	\bibliographystyle{icml2023}

\newpage
\appendix
\onecolumn
\input{Appendix}


\end{document}

%% file: macros.tex
\global\long\def\Lbf{\mathbf{L}}%

\global\long\def\Cbb{\mathbb{C}}%
\global\long\def\Ebb{\mathbb{E}}%
\global\long\def\Fbb{\mathbb{F}}%
\global\long\def\Nbb{\mathbb{N}}%
\global\long\def\Rbb{\mathbb{R}}%
\global\long\def\extR{\widebar{\mathbb{R}}}%
\global\long\def\Pbb{\mathbb{P}}%
\global\long\def\Acal{\mathcal{A}}%
\global\long\def\Bcal{\mathcal{B}}%
\global\long\def\Ccal{\mathcal{C}}%
\global\long\def\Dcal{\mathcal{D}}%
\global\long\def\Fcal{\mathcal{F}}%
\global\long\def\Gcal{\mathcal{G}}%
\global\long\def\Hcal{\mathcal{H}}%
\global\long\def\Ical{\mathcal{I}}%
\global\long\def\Kcal{\mathcal{K}}%
\global\long\def\Lcal{\mathcal{L}}%
\global\long\def\Mcal{\mathcal{M}}%
\global\long\def\Ncal{\mathcal{N}}%
\global\long\def\Ocal{\mathcal{O}}%
\global\long\def\Pcal{\mathcal{P}}%
\global\long\def\Ucal{\mathcal{U}}%
\global\long\def\Scal{\mathcal{S}}%
\global\long\def\Tcal{\mathcal{T}}%
\global\long\def\Xcal{\mathcal{X}}%
\global\long\def\Ycal{\mathcal{Y}}%
\global\long\def\Ubf{\mathbf{U}}%
\global\long\def\Pbf{\mathbf{P}}%
\global\long\def\Ibf{\mathbf{I}}%
\global\long\def\Ebf{\mathbf{E}}%
\global\long\def\Abs{\boldsymbol{A}}%
\global\long\def\Qbs{\boldsymbol{Q}}%
\global\long\def\Lbs{\boldsymbol{L}}%
\global\long\def\Pbs{\boldsymbol{P}}%
\global\long\def\i{i}%
\global\long\def\Ibb{\mathbb{I}}
\newcommand{\grad}{\nabla}
\newcommand{\gap}{\text{Gap}}
\newcommand{\tsum}{\textstyle{\sum}}

%% file: sec-abstract.tex
\begin{abstract}
	We investigate a primal-dual (PD) method for the saddle point problem (SPP) that uses a linear approximation of the primal function instead of the standard proximal step, resulting in a linearized PD (LPD) method. For convex-strongly concave  SPP, we observe that the LPD method has a suboptimal dependence on the Lipschitz constant of the primal function. 
	To fix this issue, we combine features of Accelerated Gradient Descent with the LPD method resulting in a single-loop Accelerated Linearized Primal-Dual (ALPD) method. ALPD method achieves the optimal gradient complexity when the SPP has a {\em semi-linear} coupling function. We also present an inexact ALPD method for SPPs with a general nonlinear coupling function that 
	maintains the optimal gradient evaluations of the primal parts and significantly improves the gradient evaluations of the coupling term compared to the ALPD method. We verify our findings with numerical experiments. 
	\vspace{-0.95em}
\end{abstract}

%% file: sec-intro.tex
\section{Introduction}
\label{intro}
As a class of optimization problems, the min-max saddle point problem (SPP) has attracted much attention in the optimization and machine learning literature. The SPPs contain many classical problems as a special case. E.g., we can transform convex optimization problems with smooth or nonsmooth objective functions into a min-max saddle point form. One can extend this observation to nonsmooth nonconvex problems relatively easily. Given their strong modeling power, SPPs have extensive applications in (distributionally) robust optimization and adversarial learning.

In this paper, we are interested in the following SPP
\begin{equation}\label{eq : SPP definition}
	\mathcal{L}(x,y):= \min_{x\in X}\max_{y \in Y} {f}(x) +\phi(x,y)-{g}(y),
\end{equation}
where we refer to $f$, $g$ and $\phi$ as the primal, dual and coupling functions, respectively.

The broad applicability of the SPP model has resulted in various algorithmic complexity studies in the literature. The major focus was on the computationally tractable {\em convex-concave} case, i.e., $\Lcal(\cdot, y)$ is convex in $x$ for all $y \in Y$ and $\Lcal(x, \cdot)$ is concave in $y$ for all $x \in X$. In this setting, $\max_{y \in Y} \Lcal (x, y)$ is a nonsmooth function in $x$. 
According to \citet{Nemirovski:1983problem}, subgradient descent for a black-box nonsmooth convex function achieves an $\epsilon$ optimality error in $\mathcal{O}(\frac{1}{\epsilon^2})$ subgradient evaluations. 
In a seminal work, \citet{Nesterov:2005} exploited the max-form of the problem to obtain a significantly improved gradient complexity of $\mathcal{O}(\frac{1}{\epsilon})$. 
This result broke the earlier established complexity lower bounds and is popularly known as {\em Nesterov's smoothing} technique. \citet{Nemirovski:2004} presented an Extragradient method that performs one extra gradient descent-ascent step in each iteration. 
This method can obtain an $\epsilon$ error on the stronger {\em gap function} criterion (c.f. Definition \ref{def: gap function}) using $\mathcal{O}(\frac{1}{\epsilon})$ gradient evaluations. 
Subsequently, \cite{chambolle2011first, chambolle2016ergodic, chen2014optimal} showed primal-dual (PD) type methods which remove the additional gradient descent-ascent step and maintain an $\mathcal{O}(\frac{1}{\epsilon})$  complexity when $\phi$ is a bilinear coupling.
Later, \cite{hamedani2021primal} extended it for the general convex-concave coupling functions.

The PD methods in \cite{chambolle2011first, hamedani2021primal} assume that the proximal operators of $f$ and $g$ are easy to evaluate. For the bilinear coupling term, i.e., $\phi(x, y) = y^\top Ax$,  \citet{Condat2013APS,Vu2011ASA} introduced LPD method where they used the linear approximation of $f$ in a PD method and proved the convergence of its iterates to saddle point. \citet{chambolle2016ergodic} considered the same design and showed LPD method has the convergence complexity of $\mathcal{O}(\tfrac{L_f + \|A\|}{\epsilon})$, where $L_f$ is the Lipschitz constant of $\grad f$ and $\|A\|$ is the operator norm of $A$. Observing that this dependence is not optimal in $L_f$, \citet{chen2014optimal} proposed an accelerated PD method whose complexity is of $\mathcal{O}(\sqrt{\tfrac{L_f}{\epsilon}} + \tfrac{\|A\|}{\epsilon})$ which significantly reduces the impact of Lipschitz constant $L_f$ on the complexity.

\citet{chambolle2011first, chambolle2016ergodic} also show that when  $f$ is strongly convex with modulus $\mu_f>0$ and the coupling term is bilinear, LPD method exhibits a much smaller complexity of $\mathcal{O}(\tfrac{\|A\|}{\sqrt{\mu_f \epsilon}})$, while using the exact proximal operators for $f$ and $g$. \citet{hamedani2021primal} extend similar results for {\em semi-linear} couplings (linear in $y$ only).

However, to our best knowledge, a few works study the impact of linearization of $f(x)$ when $g(y)$ is strongly convex with modulus $\mu_g>0$. \citet{kovalev2022accelerated}, showed linear convergence under a restricted strong concavity-type condition for a bilinear coupling function. \citet{thekumparampil2022lifted} proposed a single-loop algorithm called \textit{Lifted Primal-Dual method} for a bilinear SPPs under strong concavity. Moreover, \citet{thekumparampil2019efficient}, introduced a three-loop algorithm called \textit{Dual Implicit Accelerated Gradient} (DIAG) where each iteration contains an implicit step in which an AGD is run. Such problems have a direct application in the {\em Nesterov's smoothing} framework: a nonsmooth convex function  $\max_{y \in Y} f(x) + \phi(x,y)$ can be smoothened by adding a strongly concave regularizer $-g(y)$ resulting in \eqref{eq : SPP definition}. Moreover, using appropriate $Y$ and $g$, we obtain equivalent formulations of variety of (smoothened) penalty functions used in constrained optimization. Assuming the exact proximal operator of objective $f$ in such cases is quite impractical. Hence, we need to study methods that can handle its linearization. 
We intend to make contributions to this setting, i.e., $\mu_g>0$ and $f$ is linearized. See Table \ref{Table1} for a comparison of our work with the relevant literature. 
\\[2mm]
1. 
Our first contribution is to observe the subtle but important difference due to linearization. In particular, when $f$ is linearized, the case of $\mu_g >0$ is qualitatively ``harder'' than $\mu_f >0$. Hence, the LPD method exhibits a weaker complexity of $\mathcal{O}(\tfrac{L_f}{\epsilon} + \tfrac{\|A\|}{\sqrt{\mu_g\epsilon}})$ (c.f. Theorem \ref{thm: strongly-convex} and \ref{thm: strongly-convex dual}).
\\[2mm]
2. A careful observation of the above complexity yields that the LPD algorithm is unable to mitigate the impact of the primal Lipschitz constant $L_f$ when $\mu_g >0$. Hence, we seek an algorithm that can accelerate convergence in the primal. Moreover, we expand the scope of the problem to include the general nonlinear couplings. To address both questions, we imbibe elements of Nesterov's Accelerated Gradient Descent (AGD) in the PD method for general nonlinear couplings, and propose a novel single-loop Accelerated Linearized PD (ALPD) method (see Algorithm \ref{alg: Accelerated Lin PD for SPP linear g}). We show that (i) for the semi-linear coupling (linear in $x$-only), the ALPD method exhibits the complexity of $\mathcal{O}(\sqrt{\tfrac{L_f}{\epsilon}})$ which significantly improves the dependence on $L_f$ compared to the LPD method\footnote{See Remark \ref{rem:compare_hamedani} for similarity with \cite{hamedani2021primal}}; (ii) for the general coupling, it exhibits the complexity of $\mathcal{O}(\sqrt{\tfrac{L_f}{\epsilon}} + \tfrac{L_{xx}}{\epsilon})$ where $L_{xx}$ is the Lipschitz constant of $\grad_x \phi(\cdot, y)$.
\\[2mm]
3. To improve the above complexity in $L_{xx}$, we propose an Inexact ALPD method. It is a two-loop algorithm that solves a proximal problem using AGD in the inner loop while the outer loop follows a ``conceptual'' ALPD method. The Inexact ALPD method obtains an $\epsilon$-error in $\mathcal{O}(\sqrt{\tfrac{L_f}{\epsilon}})$ evaluations of $\grad f$ and $\tilde{\mathcal{O}}(\tfrac{\sqrt{L_{xx}}}{\epsilon^{3/4}})$ evaluations of $\grad_x \phi$. Essentially, this method maintains the optimal dependence of the complexity on $L_f$ and improves the dependence on $L_{xx}$.
\\[2mm]
4. We verify our findings using numerical experiments on the penalty problems 
{for linear and nonlinear constraints}.
\subsection{Related works} 
The SPPs are extensively studied in the literature due to their broad applicability and strong modeling power. Here, we provide a brief review of the most relevant first-order methods that consider the issue of algorithmic complexity for  the SPPs.
\\
{\bf Classical results:} \citet{Nesterov:2005} reformulated a deterministic optimization problem into an SPP form and showed the first optimally converging algorithm using the smoothing framework. Subsequently, \citet{Nemirovski:2004} showed the optimal convergence of \begin{table*}[t]
	\centering
	\caption{Comparison of our work. Gradient complexity is for obtaining an $\epsilon$ error in gap function.}
	\begin{tabular}{c|c|c|c|c}\hline
		&Coupling &Linearizing $f$ &\multicolumn{2}{c}{Gradient Complexity}\\\cline{4-5}
		&&&$\mu_f > 0$ &$\mu_g>0$\\\hline 
		\cite{chambolle2011first} &bilinear &No &$\mathcal{O}(\tfrac{1}{\sqrt{\epsilon}})$ &NA\\\hline
		\cite{chambolle2016ergodic} &bilinear &Yes &$\mathcal{O}(\tfrac{1}{\sqrt{\epsilon}})$ &NA\\\hline
		\cite{hamedani2021primal} &semi-linear &No &$\mathcal{O}(\tfrac{1}{\sqrt{\epsilon}})$ &NA\\\hline
		LPD (Algorithm \ref{alg: Lin PD for SPP}) &bilinear &Yes &$\mathcal{O}(\tfrac{1}{\sqrt{\epsilon}})$ & $\mathcal{O}(\tfrac{L_f}{\epsilon} + \tfrac{\|A\|}{\sqrt{\mu_g\epsilon}})$\\\hline
		\multirow{2}{*}{ALPD (Algorithm \ref{alg: Accelerated Lin PD for SPP linear g})} &semi-linear&\multirow{2}{*}{Yes}&\multirow{2}{*}{NA}&$\mathcal{O}(\sqrt{\tfrac{L_f+L_{yy}}{\epsilon}} + \tfrac{L_{xy}}{\sqrt{\mu_g\epsilon}})$ \\\cline{2-2}\cline{5-5}
		&general&&&$\mathcal{O}(\sqrt{\tfrac{L_f+L_{yy}}{\epsilon}} + \tfrac{L_{xy}}{\sqrt{\mu_g\epsilon}} + \tfrac{L_{xx}}{\epsilon})$\\\hline
		\multirow{2}{*}{Inexact ALPD (Algorithm \ref{alg: Accelerated Lin PD for SPP linear g prox})} &\multirow{2}{*}{general}   &\multirow{2}{*}{Yes} &\multirow{2}{*}{NA} &For $\grad f, \grad_y \phi : \ \mathcal{O}(\sqrt{\tfrac{L_f + L_{yy}}{\epsilon}})$\\
		&&&&For $\grad_x\phi:\ \mathcal{O}(\tfrac{\sqrt{L_{xx}\sqrt{L_{f}+L_{xy}^2/\mu_g }} }{\epsilon^{3/4}}\log(\tfrac{1}{\epsilon}))$\\\hline
	\end{tabular}
	\label{Table1}
\end{table*}
the mirror-prox method (a generalization of the extragradient method \cite{korpelevich1976extragradient}) for the variational inequality problem which contains the nonlinear SPP as a special case. Separately, \citet{nesterov2007dual} and \citet{tseng2008accelerated} provided two optimally converging algorithms for the SPPs. This approach was further extended by \citet{monteiro2010complexity} in an HPE framework to relax the bounded domain assumption. \citet{Nemirovski:2009robust} presented a mirror-descent type algorithm for the stochastic SPP. \citet{Juditsky:2011svi} proposed a stochastic version of the mirror-prox method. \citet{chen2017accelerated} incorporated a multi-step acceleration scheme into the stochastic mirror-prox to improve the convergence rate.  
\\
{\bf Bilinear case:} While extragradient (or mirror-prox) required two $\grad_x, \grad_y$ evaluations in each iteration, the primal-dual method of \cite{chambolle2011first} required only one such evaluation per iteration and maintained the same convergence rate. 
Several variants of this method are proposed in the literature for bilinear couplings. E.g., the linearization of $f$ is presented in \cite{chambolle2016ergodic}, optimal accelerated-version is introduced in \cite{chen2014optimal}, randomized block-coordinate settings are considered in \cite{dang2014randomized, zhu2015adaptive, yu2015doubly, zhang2015stochastic}.
\\
{\bf Nonlinear coupling:} For the nonlinear coupling term, \citet{hamedani2021primal} proposed a primal-dual method which can be seen as an extension of the original primal-dual method. Its extension to randomized block-coordinate version was presented in \cite{hamedani2018iteration}. Another variation of significant consequence is proposed in \cite{boob2022stochastic} for the  stochastic smooth/nonsmooth function-constrained optimization.
\\
{\bf Strong convexity:} To our best knowledge, the existing works look at the strongly convex case ($\mu_f >0$). For the  bilinear couplings, \citet{chambolle2011first} shows a smaller complexity of  $\mathcal{O}(\frac{1}{\sqrt{\epsilon}})$. \citet{hamedani2021primal} present the first accelerated convergence result for semi-linear coupling (linear in $y$-only). \citet{lin2020near} proposed an inexact accelerated proximal point algorithm which has a nested three-loop structure and obtains an optimal complexity up to a $\log^3(\tfrac{1}{\epsilon})$ factor. The problem of obtaining optimal rates for general nonlinear couplings with single-loop algorithms remains open.

%% file: sec-tech-overview.tex
\section{Technical overview - The LPD method}\label{sec:Technical Overview}

For the bilinear SPP, i.e., $\phi(x,y) = y^\top Ax$, most PD methods use computationally expensive  proximal operators of $f$ and $g$. This may be reasonable in some applications where $g$ is a regularizing function.
However, that is not the case for $f$ which arises from the primal optimization. To overcome this challenge, the linearized PD method \cite{chambolle2016ergodic} uses a linear approximation $f(x_t) + \langle\,\nabla f(x_t),x-x_t\rangle$ instead of evaluating a proximal operator.
%
Algorithm \ref{alg: Lin PD for SPP} illustrates a typical LPD method, where parameters ${\tau_t}$ and ${\eta_t}$ denote the step-sizes (or learning rates) in the dual and primal updates, respectively. The {\em momentum} parameter $\theta_t$ is used to generate an extrapolated sequence $\{\tilde{x}_t\}$  which is then used for the accelerated update of the dual $y$ (line 3). On the other hand, the method uses a simple gradient descent step to update $x$ (line 4). The algorithm outputs an ergodic average after $K$ iterations.
\citet{chambolle2016ergodic} showed an accelerated convergence of $\Ocal(\tfrac{1}{K^2})$ for the strongly convex case $(\mu_f > 0, \mu_g = 0)$. However, the strongly concave case $(\mu_f = 0, \mu_g > 0)$ is missing. Furthermore, it is important to note that the two cases are not symmetric since we are linearizing the primal function $f$. A closer inspection shows that the two cases are quantitatively different.
Here, we present two contrasting (and hence, somewhat surprising) results for the LPD method for these cases. Theorem \ref{thm: strongly-convex} considers $\mu_f>0$, and show convergence rate of $\mathcal{O}(\tfrac{1}{K^2})$ for the LPD method \footnote{Though the result is similar to \cite{chambolle2016ergodic}, the step-size policy is significantly different.}. 
However, the LPD method does not effectively handle the error caused by the linearization of $f$ when $\mu_g > 0$ (see Theorem \ref{thm: strongly-convex dual}). 
Below, we state the step-size conditions required for the analysis of the LPD method. See Appendix \ref{apx:LPD} for proofs of all results in this section.

\textbf{Step-size conditions for the LPD method:} 
For $t\geq 2$
\vspace{-5pt}
\begin{subequations}\label{eq : conditions PD}
	\begin{align}
		\gamma_{t+1}(\tfrac{1}{\eta_t}-\mu_f)&\leq \tfrac{\gamma_{t}}{\eta_{t-1}}, \label{eq:conditionPD1}\\[-3pt]
		\tfrac{\gamma_{t+1}}{\tau_t} &\leq \gamma_{t}\left(\mu_g+\tfrac{1}{\tau_{t-1}}\right), \label{eq:conditionPD2}\\[-3pt]
		\theta_{t-1} &=\tfrac{\gamma_{t}}{\gamma_{t+1}},\label{eq:conditionPD3}\\[-3pt]
		\theta_{t-1}\|A\|^2 &\leq (\tfrac{1}{\eta_{t-1}}-L_{{f}})\tfrac{1}{\tau_t}.\label{eq:conditionPD4}
	\end{align}
\end{subequations}
\begin{algorithm}[t]
	\caption{  Linearized PD (LPD) method }\label{alg: Lin PD for SPP}
	\begin{algorithmic}[1]
		\STATE {\bf Initialize} $\tilde{x}_1=x_1 \in X,\ y_1 \in Y$
		\FOR{$t = 1, \ldots, K$}
		\STATE $y_{t+1} \gets\arg\min\limits_{y\in Y} \langle\,-A\tilde{x}_t,y\rangle + {g}(y)+\tfrac{1}{2\tau_t}\|y-y_t\|^2$
		\STATE $x_{t+1}\gets\arg\min\limits_{x\in X} \langle\,\grad f(x_t)+A^\top y_{t+1},x \rangle  +\tfrac{1}{2\eta_t}\|x-x_t\|^2$
		\STATE $\tilde{x}_{t+1}\gets x_{t+1}+\theta_t(x_{t+1}-x_t)$
		\ENDFOR
		\STATE {\bf return} $\bar{x}_{K+1} \gets \tfrac{\sum_{t=1}^K \gamma_{t+1}x_{t+1}}{\sum_{t=1}^K \gamma_{t+1}},\bar{y}_K \gets \tfrac{\sum_{t=1}^K \gamma_{t+1}y_{t+1}}{\sum_{t=1}^K \gamma_{t+1}}$
	\end{algorithmic}
\end{algorithm}

\vspace{-1em}
\begin{theorem}\label{thm: strongly-convex}
	Assume that $\mu_f>0$, $\mu_g =0$ and set parameters $\{\gamma_t, \theta_t, \eta_t, \tau_t\}$ as per the following:
	\vspace{-5pt}
	\begin{equation}\label{eq:strong_convex_step_policy}
		\begin{split}
			\gamma_{t} = \tfrac{t}{2} + \tfrac{L_{{f}}}{\mu_f} , \qquad &\theta_{t-1} = \tfrac{t/2 + L_{{f}}/\mu_f }{(t+1)/2 + L_{{f}}/\mu_f},\\ 
			\tfrac{1}{\eta_t} = \mu_f \tfrac{t+1}{2} + L_{f}, \quad &\tfrac{1}{\tau_t} = \tfrac{4\|A\|^2}{\mu_f(t+1)/2}.
		\end{split}
	\end{equation} 
	Then, we have	
	\begin{align}
			\gap(\bar{z}_{K+1})\leq& 
			\tfrac{4}{K(K+3+4L_{{f}}/{\mu_f})} \big[(1+\tfrac{L_{{f}}}{\mu_f})[\tfrac{\mu_f+L_{{f}}}{2} \|x-x_1\|^2\nonumber\\
			&+\tfrac{4\|A\|^2}{2\mu_f}\|y-y_1\|^2]\big]. \label{eq: gap thm1}
	\end{align}
\end{theorem}
It is easy to see that the step-size policy \eqref{eq:strong_convex_step_policy} satisfies the  conditions in \eqref{eq : conditions PD}. Theorem \ref{thm: strongly-convex} shows  $\Ocal(\tfrac{1}{K^2})$ convergence rate for Algorithm \ref{alg: Lin PD for SPP}. It is also interesting to note that \eqref{eq:strong_convex_step_policy} provides an explicit expression of the weights $\gamma_t$ which results in an explicit bound of  $\Theta(K^2)$ on $\tsum_{t=1}^K \gamma_t$ for $K\ge 1$. This bound is usually shown implicitly and for only large values of $K$ in \cite{chambolle2011first, chambolle2016ergodic, hamedani2021primal}. For the semi-linear couplings, a similar explicit policy is used in \cite{boob2022stochastic}.

In the second case  ($\mu_f = 0 , \mu_g>0$), however, a step-size approach similar to \eqref{eq:strong_convex_step_policy} is not applicable. The following argument provides a rather \underline{mechanical intuition}: To have an accelerated convergence rate of $\mathcal{O}(\tfrac{1}{K^2})$, we need $\Gamma_K := \tsum_{t=1}^K\gamma_t = \Omega(K^2)$ and hence $ \gamma_{t}$ needs to increase linearly in $t$. In view of $\mu_f =0$, \eqref{eq:conditionPD1} requires $\tfrac{\gamma_{t+1}}{\eta_{t}}$ to be a decreasing sequence and we get $\tfrac{\gamma_{2}}{\eta_1} \ge \tfrac{\gamma_{K+1}}{\eta_{K}}$.  Simultaneously, to mitigate errors generated by linearization of $f$, we require $\tfrac{1}{\eta_K} \ge L_f$ (see \eqref{eq:conditionPD4}). These two relations and linearly increasing nature of $\gamma_{t}$ imply that $\tfrac{1}{\eta_1} \ge  \tfrac{L_f \gamma_{K+1}}{\gamma_{2}} = \Omega(L_fK)$. This is problematic since the final convergence error of the LPD method is of $\mathcal{O}(\tfrac{\gamma_{2}}{\eta_1 \Gamma_K})=\mathcal{O}(\tfrac{L_f}{K})$, a weaker convergence compared to $\mathcal{O}(\tfrac{1}{K^2})$.
This is not observed when $\mu_f > 0$ and $\mu_g = 0$.
Indeed in \eqref{eq:strong_convex_step_policy}, we see that both $\gamma_t$ and $\tfrac{1}{\eta_t}$ are both increasing in $t$ and still \eqref{eq:conditionPD1} is satisfied. 

The critical issue is that \eqref{eq:conditionPD1} requires $\{\tfrac{\gamma_{t+1}}{\eta_t}\}$ to be a decreasing sequence when $\mu_f = 0$. To provide a principled solution to this problem, we modify \eqref{eq:conditionPD1} to allow $\frac{\gamma_{t+1}}{\eta_t}$ to increase with $t$ by a fixed amount (see \eqref{eq: new condition}). This approach requires a new step-size policy discussed below.\\
{\bf Modified step-size condition for the LPD method:} Modify \eqref{eq:conditionPD1} as follows while keeping \eqref{eq:conditionPD2}-\eqref{eq:conditionPD4} unchanged:
\begin{equation}\label{eq: new condition}
	\tfrac{\gamma_{t+1}}{\eta_{t}}  -\tfrac{\gamma_{t}}{\eta_{t-1}}  \le L_{{f}}
\end{equation}
\begin{theorem}\label{thm: strongly-convex dual}	
	Suppose $\mu_g>0, \mu_f =0$ and set parameters $\{ \gamma_t, \theta_t, \eta_t, \tau_t\}$ as per the following:
	\vspace{-5pt}
	\begin{equation}\label{eq:strong_concave_step_policy}
		\setlength{\jot}{-1pt}
		\begin{split}
			\gamma_{t} = t,\quad &\tfrac{1}{\tau_{t}} = \mu_g \tfrac{t}{2},\\
			\tfrac{1}{\eta_{t}} = \tfrac{2\|A\|^2}{\mu_g(t+1)} + L_{{f}} ,\quad &\theta_{t-1} =\tfrac{t}{t+1}.
		\end{split}
	\end{equation} Then, we have
	\begin{equation}
		\begin{aligned}
			\gap(\bar{z}_{K+1})\leq& \tfrac{2D_x^2\|A\|^2/\mu_g  + D_y^2\mu_g}{K^2} + \tfrac{2(K+1)L_{{f}} D_x^2}{K^2}.
		\end{aligned}
	\end{equation}
\end{theorem}	
Note that \eqref{eq:strong_concave_step_policy} satisfies the modified step-size condition \eqref{eq: new condition} and \eqref{eq:conditionPD2}-\eqref{eq:conditionPD4}. From the result, it is clear that for the strongly concave SPP ($\mu_g > 0$), the convergence rate of the LPD method is of $\mathcal{O}(\tfrac{\|A\|^2}{K^2} + \tfrac{L_f}{K})$ when $f$ is linearized.

This result is in sharp contrast with Theorem \ref{thm: strongly-convex} where the convergence rate is of $\mathcal{O}(\tfrac{1}{K^2})$. We already provided a mechanical reasoning for the ineffectiveness of the LPD method in reducing the impact of Lipschitz constant $L_f$. At a broader design level, the algorithm itself is not accelerated in the primal iterate. Indeed, it is simply a gradient descent in the $x$-update (see Line 4 in Algorithm \ref{alg: Lin PD for SPP}). This was not a problem when $f$ was strongly convex. However, when only the dual is strongly-concave, one needs a stronger acceleration in the primal to mitigate the errors caused by the linearization of $f$.
Hence, the rest of this paper is dedicated to presenting the accelerated linearized PD algorithm and its variant for obtaining more robust convergence results for problem \eqref{eq : SPP definition} when $f$ is linearized and $\mu_g > 0$.

%% file: sec-numerical-ex.tex
\section{Numerical Experiments} \label{sec:numericalex}
In this section, we perform numerical experiments to (i) compare the performance of the LPD and ALPD algorithms on the penalty problems with different settings; (ii) evaluate the runtime performance of the step-size policies in Theorem \ref{thm: strongly-convex} and \cite{chambolle2016ergodic}; (iii) compare the ALPD and Inexact ALPD on penalty problems for nonlinear constraints. All experiments are performed on 64-bit Windows 10 with Intel i5-9500U @3.00GHz and 16GB RAM. 
\subsection{ALPD vs. LPD} \label{sec:ALPD_vs_LPD}
The $\ell_q$-norm penalty problem with linear constraints is
\begin{equation*}
	\begin{aligned}
		 \min_{ x\in X} f(x)+ \rho \|Ax-b\|_q \equiv \min_{ x\in X} \max_{ \|y\|_p\leq 1} f(x) + \rho \langle y, Ax-b\rangle,
	\end{aligned}
\end{equation*}
where $\ell_p$-norm is the dual norm of $\|\cdot\|_q$. The equivalence of the dual formulation is well-known where $1/p + 1/q = 1$.
We can get a smooth approximation of the nonsmooth penalty term using Nesterov's smoothing technique
\begin{equation}\label{eq:quad_lin_SPP}
	\begin{aligned}
		\min_{ x\in X}\max_{ \|y\|_p\leq 1}&\{f(x) + \rho\langle\,y,Ax-b\rangle - \tfrac{\mu_g}{2}\|y\|^2\},
	\end{aligned}
\end{equation}
where parameter $\mu_g$ can be used to calibrate the smoothness of the approximation.
We set $f(x) = \frac12 x^\top Qx + c^\top x$ as a convex quadratic function where $Q \in \Rbb^{n\times n}$ is a randomly generated positive semidefinite matrix and $c \in \Rbb^n$ is a random vector.  We also generate matrix $A \in \Rbb^{m\times n}$ and $b \in \Rbb^m$ randomly. For these experiments, we set the penalty parameter $\rho = 1$ and $m = n = 100$. Appendix \ref{sec:detailedfata} provides the detailed information on the exact functions used for the random number generation. We set $L_f = 200$ since eigenvalues of $Q$ are generated uniformly on $[0, 200]$. 

We implement two versions of the ALPD method. The first method is implemented exactly as presented in Algorithm \ref{alg: Accelerated Lin PD for SPP linear g}. The second method uses a proximal operator of $g$ as follows: line 5 of Algorithm \ref{alg: Accelerated Lin PD for SPP linear g} is replaced by $y_{t+1} = \arg\min_{y \in Y} \langle-v_t ,y\rangle +g(y) +\tfrac{1}{2\tau_{t}}\|y-y_t\|^2$. We make these changes (1) to measure the effect of using linearization in $g$ on the numerical performance of the ALPD method, and (2) to perform a fair comparison with the LPD method as it uses the  more advantageous proximal operator of $g$. We refer to this method as ALPD-prox-g. The step-size policy for this version is similar to \eqref{stepsize: case1} with $L_g = 0$ since $g$ is used exactly without linearization. We measure the performance of the algorithms using three metrics: (1) Gap function which is the standard metric used in the convergence analysis, (2) Primal relative error $\|\bar{x}_t-x^*\|/ \|x^*\|$, and (3) Dual relative error $\| \bar{y}_t -y^*\|/ \| y^*\|.$ All algorithms start at the same randomly generated initial point in the domain $X \times Y$. Figure \ref{fig:ALPD_LPD_comp_QUAD_LIN} compares the three algorithms in three metrics. Each plot is generated using the average performance of the algorithms on 10 instances of \eqref{eq:quad_lin_SPP} generated independently with identically distribution (i.i.d. instances). We plot the metrics for the last 50 iterations to focus on the  major performance differences.  
\begin{figure}[t]
	\centering
	\includegraphics[width=1\linewidth,]{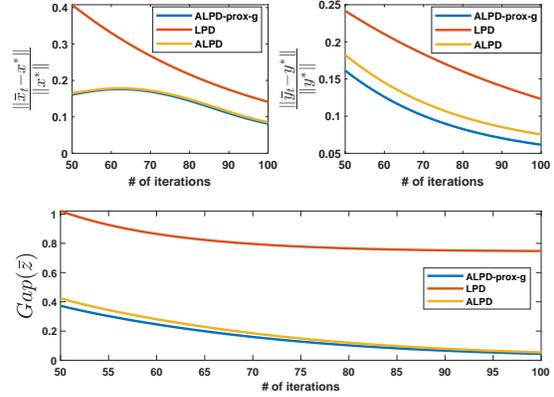}
	\caption{Comparison of the methods in terms of the mean errors in primal (top left), dual (top right) and Gap function (bottom) for 10 i.i.d. instances of \ref{eq:quad_lin_SPP} with $p=q=2$.}
	\label{fig:ALPD_LPD_comp_QUAD_LIN}
\end{figure} 
 Figure \ref{fig:ALPD_LPD_comp_QUAD_LIN} shows that when $L_f =200$ (a large number), the LPD method performs poorly compared to both versions of ALPD. Moreover, ALPD-prox-g gives a slight advantage over ALPD which is expected. Note that in these experiments, we use $p = q =2$. In Appendix \ref{sec:more norms}), we provide a similar comparison for two settings of \eqref{eq:quad_lin_SPP}: $q=1$ and $q = \infty$. 
 \subsection{ALPD vs. Inexact ALPD}
 In this subsection, we compare the performances of Algorithms \ref{alg: Accelerated Lin PD for SPP linear g} and \ref{alg: Accelerated Lin PD for SPP linear g prox} on the penalty problem with nonlinear constraints. We replace the linear constraints in the previous case by quadratic constraints $ \tfrac{1}{2} x^\top A_j x + b_j^\top x -d_j\leq 0,$ for all $j \in [m] $ where $A_j, b_j$ and $d_j (> 0)$ are randomly generated as in the previous experiment. The dual form of the penalty functions on nonlinear constraints has $L_{xx} > 0$. As we proved in Section \ref{sec: step size policy AGD}, when $L_{xx}>0$, Inexact ALPD is superior to ALPD in terms of gradient complexity. To verify our results, we run 10 i.i.d. instances of the nonlinear penalty problem and plot the Gap function against the average run time of each algorithm. For the ALPD method, we use the step-size policy in Section \ref{sec:nonlinear_stepsize} and Inexact ALPD method is employed as described in Algorithm \ref{alg: Accelerated Lin PD for SPP linear g prox}. Moreover, we implement prox versions of both algorithms where we use the proximal oracle of $g$ instead of linearizing it. We call these versions ALPD-prox-g and Inexact-ALPD-prox-g respectively. Figure \ref{fig:ALPD_ALPD_INEXACT} illustrates the behavior of these algorithms for 100-dimensional ($ n=100 $) penalty problems with 10 non-linear constraints ($ m=10 $). We run the ALPD method for 200 iterations and its inexact counterpart for 100 iterations. We can see that Inexact ALPD and Inexact ALPD-prox-g dominates the performance of ALPD and ALPD-prox-g, respectively.
 \begin{figure}[h]
 	\centering
 	\includegraphics[scale=0.45]{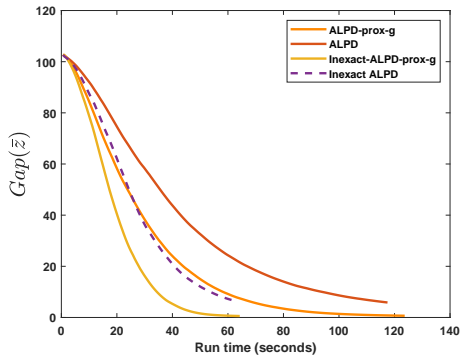}
 	\caption{Comparison of the ALPD and inexact ALPD method and their prox-g variants using the Gap function vs run-time (seconds) plot for 10 i.i.d. instances.}
 	\label{fig:ALPD_ALPD_INEXACT}
 \end{figure} 
\subsection{LPD step-size policy comparison} \label{sec:chambollevsours}
As we mentioned in Section \ref{sec:Technical Overview}, both policies in \eqref{eq:strong_convex_step_policy} and \cite{chambolle2016ergodic} give similar convergence rates asymptotically. To make the numerical comparison, we use the SPP in \eqref{eq:quad_lin_SPP} with $\mu_g = 0$. We set $\mu_f$ as the minimum eigenvalue of the randomly generated matrix $Q$. Note that $\mu_f > 0$ almost surely. We run the LPD method for 10 i.i.d. instances of this problem for each step-size policy. See Appendix~\ref{apx:LPD_comparison} for the details of our numerical study. It seems that the LPD method using step-size in \eqref{eq:strong_convex_step_policy} performs better than \cite{chambolle2016ergodic}. We conjecture the following reason for this deviation in the performance: \citet{chambolle2016ergodic}  show that $\tsum_{t=1}^K\gamma_t = \Omega(K^2)$ only for large values of $K$ whereas \eqref{eq:strong_convex_step_policy} defines $\gamma_{t} = \Theta(t)$ explicitly and hence $\tsum_{t=1}^K\gamma_t = \Theta(K^2)$ for all $K \ge 1$. The quadratic growth of $\sum_{t=1}^K\gamma_{t}$ is important to obtain the accelerated $O(\tfrac{1}{K^2})$ convergence rate. Hence, the step-size policy in \eqref{eq:strong_convex_step_policy} seems to be performing well in our experiments.

%% file: Appendix.tex
\section*{Appendix}
\section{General Analysis of Algorithm \ref{alg: Lin PD for SPP} (LPD) } \label{apx:LPD}
In this section, we state some technical results that are ultimately used for obtaining \eqref{eq:conditionPD1}-\eqref{eq:conditionPD4} and Theorems \ref{thm: strongly-convex} and \ref{thm: strongly-convex dual}. \\
First, let us state two important lemmas that are utilized in the rest the discussion specially when we want to construct relations related to optimality points. 
\begin{lemma} \label{lemma:threepoint}
	Let $x^{\star}$ be a $\delta$-approximate solution of problem $\min_{x \in X}\{h(x) + \tfrac{\lambda}{2}\|x-\hat{x}\|^2\}$ where $h(x)$ is a convex function. Then, 
	\begin{equation}
		\begin{aligned}
			h(x^{\star}) - h(x) \leq& \tfrac{\lambda}{2}\big[\|x-\hat{x}\|^2-\|x^{\star}-x\|^2 -\|x^{\star}-\hat{x}\|^2\big] + \delta + \sqrt{2\lambda\delta}\|x^{\star}-x\|.
		\end{aligned}
	\end{equation}  
	This lemma is known as "Three-point" lemma and also can be stated for a strongly-convex function $h$ with modulus $\mu_h$ as below
	\begin{equation}
		\begin{aligned}
			h(x^{\star}) - h(x) \leq& \tfrac{\lambda}{2}\big[\|x-\hat{x}\|^2-\|x^{\star}-x\|^2\ -\|x^{\star}-\hat{x}\|^2\big] - \tfrac{\mu_h}{2}\|x^{\star}-x\|^2
			+ \delta + \sqrt{2\lambda\delta}\|x^{\star}-x\|.
		\end{aligned}
	\end{equation} 
\end{lemma}   
\begin{lemma} \label{lemma: lemma 1}
	For point $z_{t+1} = (x_{t+1},y_{t+1}) \in Z$ in Algorithm \ref{alg: Lin PD for SPP}, the primal-dual gap function is upper bounded as follows
	\begin{equation}\label{eq: lemma 1.1}      
		\begin{aligned}
			Q(z_{t+1},z)\leq &\tfrac{L_{{f}}}{2}\|x_{t+1}-x_t\|^2-\tfrac{\mu_f}{2}\|x-x_t\|^2
			+\langle\,\nabla {f}(x_t),x_{t+1}-x \rangle+ \left[ {g}(y_{t+1})-{g}(y)\right] 
			+ \langle\,Ax_{t+1},y\rangle - \langle\,Ax,y_{t+1}\rangle.
		\end{aligned}
	\end{equation}
\end{lemma}
\begin{proof} 
	Since $ f $ is $L_{{f}}$-smooth, we have
	\begin{equation*}
		\begin{aligned}
			{f}(x_{t+1})
			& \leq{f}(x_{t})+ \langle\,\nabla {f}(x_t),x_{t+1}-x_{t} \rangle+\tfrac{L_{{f}}}{2}\|x_{t+1}-x_t\|^2\\
			& = {f}(x_{t})+\langle\,\nabla {f}(x_t),x_{t+1}-x \rangle +\langle\,\nabla {f}(x_t),x-x_{t} \rangle +\tfrac{L_{{f}}}{2}\|x_{t+1}-x_t\|^2\\
			&\leq {f}(x)+ \langle\,\nabla {f}(x_t),x_{t+1}-x \rangle+\tfrac{L_{{f}}}{2}\|x_{t+1}-x_t\|^2
			-\tfrac{\mu_f}{2}\|x-x_t\|^2.
		\end{aligned}
	\end{equation*}	
	Adding $[{g}(y_{t+1})-{g}(y)]$ , and $[\langle\,Ax_{t+1},y\rangle-\langle\,Ax,y_{t+1}\rangle]$ to the both sides leads to the \eqref{eq: lemma 1.1}.
\end{proof}
We can elaborate on the upper bound by using the optimality conditions of $y_{t+1}$ and $x_{t+1}$ respectively. The following theorem illustrates a useful upper bound for the weighted gap function for the LPD method.

\begin{theorem}\label{thm: thm1}
	
	if for $t\geq 2$
	\vspace{-5pt}
	\begin{subequations}\label{eq : conditions PD1}
		\begin{align}
			\gamma_{t+1}(\tfrac{1}{\eta_t}-\mu_f)&\leq \tfrac{\gamma_{t}}{\eta_{t-1}}, \label{eq:condition11}\\
			\tfrac{\gamma_{t+1}}{\tau_t} &\leq \gamma_{t}\left(\mu_g+\tfrac{1}{\tau_{t-1}}\right), \label{eq:condition21}\\
			\theta_{t-1} &=\tfrac{\gamma_{t}}{\gamma_{t+1}},\label{eq:condition31}\\
			\theta_{t-1}\|A\|^2 &\leq (\tfrac{1}{\eta_{t-1}}-L_{{f}})\tfrac{1}{\tau_t}.\label{eq:condition41}
		\end{align}
	\end{subequations}
	then
	\begin{equation}\label{eq: THM1.1}
		\setlength{\jot}{-1pt}
		\begin{aligned}
			\sum_{t=1}^{K}\gamma_{t+1}Q(z_{t+1},z)
			& \leq \tfrac{\gamma_{2}}{2} (\tfrac{1}{\eta_1} -\mu_f) \|x-x_1\|^2-\tfrac{\gamma_{K+1}}{2\eta_{K}}\|x-x_{K+1}\|^2
			+ \tfrac{\gamma_{2}}{2\tau_1}\|y-y_1\|^2\\
			&\quad-\tfrac{\gamma_{K+1}}{2}\left(\mu_g+\tfrac{1}{\tau_{K}}-\tfrac{\|A\|^2}{\tfrac{1}{\eta_{K}}-L_{{f}}}\right)\|y-y_{K+1}\|^2,
		\end{aligned}
	\end{equation}
	and at optimality
	\vspace{-5pt}
	\begin{equation*}
		\begin{aligned}
			\tfrac{\gamma_{K+1}}{2}\left(\mu_g+\tfrac{1}{\tau_{K}}-\tfrac{\|A\|^2}{\tfrac{1}{\eta_{K}}-L_{{f}}}\right)\|y-y_{K+1}\|^2
			&\leq\tfrac{\gamma_{2}}{2\eta_1} \|x^\star-x_1\|^2
			+\tfrac{\gamma_{2}}{2\tau_1} \|y^\star-y_1\|^2.
		\end{aligned}
	\end{equation*}		
\end{theorem}
\begin{proof}
	Using the optimality of $y_{t+1}$ and Lemma \ref{lemma:threepoint} where $\delta=0$ (note $y_{t+1}$ is an exact solution.), we have
	\begin{equation}\label{eq: optim dualbi}
		\begin{aligned}
			{g}(y_{t+1})-{g}(y)\leq &\tfrac{1}{2\tau_t}\left( \|y-y_t\|^2-\|y_{t+1}-y_t\|^2\right)
			-(\tfrac{1}{2\tau_{t}}+\tfrac{\mu_g}{2})\|y-y_{t+1}\|^2+\langle\,A\tilde{x_t},y_{t+1}-y\rangle.
		\end{aligned}
	\end{equation}
	Also, from the optimality of $x_{t+1}$, we have the following
	\begin{equation}\label{eq: optim prim}
		\begin{aligned}
			\langle\,\nabla f(x_t),x_{t+1}-x \rangle
			&\leq \tfrac{1}{2\eta_t}\left( \|x-x_t\|^2-\|x_{t+1}-x_t\|^2\right)
			-\tfrac{1}{2\eta_t} \|x-x_{t+1}\|^2-\langle\,A(x_{t+1}-x),y_{t+1}\rangle.
		\end{aligned}
	\end{equation}
	From \eqref{eq: lemma 1.1}, \eqref{eq: optim dualbi} and \eqref{eq: optim prim}, one can reconstruct the following upper bound on the gap function at one iteration
	\begin{equation}\label{eq: 2nd upbound}
		\begin{aligned}
			Q(z_{t+1},z)
			&\leq\Big((\tfrac{1}{2\eta_t}-\tfrac{\mu_f}{2})\|x-x_t\|^2 - \tfrac{1}{2\eta_t}\|x-x_{t+1}\|^2\Big) -(\tfrac{1}{2\eta_t}-\tfrac{L_f}{2})\|x_{t+1}-x_t\|^2+\tfrac{1}{2\tau_t}\left( \|y-y_t\|^2-\|y_{t+1}-y_t\|^2\right)\\
			&\quad-(\tfrac{1}{2\tau_{t}}+\tfrac{\mu_g}{2})\|y-y_{t+1}\|^2-\langle\,A(x_{t+1}-x),y_{t+1}\rangle
			+\langle\,A\tilde{x_t},y_{t+1}-y\rangle+ \langle\,Ax_{t+1},y\rangle - \langle\,Ax,y_{t+1}\rangle.\\
		\end{aligned}
	\end{equation}
	We can simplify the upper bound with respect to the inner products. 
	\begin{equation*}
		\begin{aligned}
			&-\langle\,A(x_{t+1}-x),y_{t+1}\rangle+\langle\,A\tilde{x_t},y_{t+1}-y\rangle
			+ \langle\,Ax_{t+1},y\rangle- \langle\,Ax,y_{t+1}\rangle\\
			&=-\langle\,A(x_{t+1}-x),y_{t+1}\rangle
			+\langle\,A(x_t+\theta_{t-1}(x_t-x_{t-1})),y_{t+1}-y\rangle
			+ \langle\,Ax_{t+1},y\rangle - \langle\,Ax,y_{t+1}\rangle\\
			&=-\langle\,Ax_{t+1},y_{t+1}\rangle+\langle\,Ax,y_{t+1}\rangle
			+\langle\,Ax_t,y_{t+1}\rangle-\langle\,Ax_t,y\rangle
			+\theta_{t-1}\langle\,A(x_t-x_{t-1}),y_{t+1}-y\rangle
			+ \langle\,Ax_{t+1},y\rangle - \langle\,Ax,y_{t+1}\rangle\\
			&=-\langle\,Ax_{t+1},y_{t+1}\rangle+\langle\,Ax_t,y_{t+1}\rangle-\langle\,Ax_t,y\rangle
			+\theta_{t-1}\langle\,A(x_t-x_{t-1}),y_{t+1}-y\rangle
			+ \langle\,Ax_{t+1},y\rangle \\
			&=-\langle\,A(x_{t+1}-x_t),y_{t+1}-y\rangle
			+\theta_{t-1}\langle\,A(x_t-x_{t-1}),y_{t+1}-y\rangle.
		\end{aligned}
	\end{equation*}
	Also, we can write the above expression as follows
	\begin{equation*}
		\begin{aligned}
			&-\langle\,A(x_{t+1}-x_t),y_{t+1}-y\rangle+\theta_{t-1}\langle\,A(x_t-x_{t-1}),y_{t+1}-y\rangle\\
			&=-[ \langle\,A(x_{t+1}-x_t),y_{t+1}-y\rangle-\theta_{t-1}\langle\,A(x_t-x_{t-1}),y_{t}-y\rangle
			+\theta_{t-1}\langle\,A(x_t-x_{t-1}),y_{t}-y_{t+1}\rangle ].
		\end{aligned}
	\end{equation*}
	From \eqref{eq: 2nd upbound},  we can rewrite the upper bound for gap function as follows
	\begin{equation}
		\begin{aligned}
			Q(z_{t+1},z) &\leq(\tfrac{1}{2\eta_t}-\tfrac{\mu_f}{2})\|x-x_t\|^2-\tfrac{1}{2\eta_t}\|x-x_{t+1}\|^2+\tfrac{1}{2\tau_t}\|y-y_t\|^2
			-(\tfrac{1}{2\tau_{t}}+\tfrac{\mu_g}{2})\|y-y_{t+1}\|^2-\langle\,A(x_{t+1}-x_t),y_{t+1}-y\rangle\\
			&\quad+\theta_{t-1}\langle\,A(x_t-x_{t-1}),y_{t}-y\rangle-\tfrac{1}{2\tau_t}\|y_{t+1}-y_t\|^2
			-(\tfrac{1}{2\eta_t}-\tfrac{L_f}{2})\|x_{t+1}-x_t\|^2
			-\theta_{t-1}\langle\,A(x_t-x_{t-1}),y_{t}-y_{t+1}\rangle.
		\end{aligned}
	\end{equation}
	Hence, multiplying both sides by $\gamma_{t+1}$ and summing up till $K$ gives us an upper bound for the average gap function for LPD. We have
	\begin{equation}\label{eq: sum bound}
		\begin{aligned}
			\sum_{t=1}^{K}\gamma_{t+1}Q(z_{t+1},z)\leq &\sum_{t=1}^{K}\gamma_{t+1}  {\Large [}(\tfrac{1}{2\eta_t}-\tfrac{\mu_f}{2})\|x-x_t\|^2-\tfrac{1}{2\eta_t}\|x-x_{t+1}\|^2		+\tfrac{1}{2\tau_t}\|y-y_t\|^2-(\tfrac{1}{2\tau_{t}}+\tfrac{\mu_g}{2})\|y-y_{t+1}\|^2{\Large ]} \\
			&\quad-\sum_{t=1}^{K}\gamma_{t+1}[\langle\,A(x_{t+1}-x_t),y_{t+1}-y\rangle
			+\theta_{t-1}\langle\,A(x_t-x_{t-1}),y_{t}-y\rangle]\\
			&\quad-\sum_{t=1}^{K}\gamma_{t+1}[\tfrac{1}{2\tau_t}\|y_{t+1}-y_t\|^2
			+(\tfrac{1}{2\eta_t}-\tfrac{L_f}{2})\|x_{t+1}-x_t\|^2
			-\theta_{t-1}\langle\,A(x_t-x_{t-1}),y_{t}-y_{t+1}\rangle].\\
		\end{aligned}
	\end{equation}
	To simplify each summation in \eqref{eq: sum bound}, let us start with the first one
	\begin{equation}\label{eq:first summation}
		\begin{aligned}
			\sum_{t=1}^{K}\gamma_{t+1}  {\Large [}(\tfrac{1}{2\eta_t}-\tfrac{\mu_f}{2})\|x-x_t\|^2-\tfrac{1}{2\eta_t}\|x-x_{t+1}\|^2		+\tfrac{1}{2\tau_t}\|y-y_t\|^2-(\tfrac{1}{2\tau_{t}}+\tfrac{\mu_g}{2})\|y-y_{t+1}\|^2{\Large ]}
		\end{aligned}
	\end{equation}
	If we assume that for each $t\geq 2$, we have \eqref{eq:condition11} and \eqref{eq:condition21}, the above summation \eqref{eq:first summation} is upper bounded by  
	\begin{equation}\label{eq: first term}
		\begin{aligned}
			\leq& \gamma_{2}(\tfrac{1}{2\eta_1}-\tfrac{\mu_f}{2})\|x-x_1\|^2-\gamma_{K+1}\tfrac{1}{2\eta_K}\|x-x_{K+1}\|^2
			+\gamma_{2}\tfrac{1}{2\tau_1}\|y-y_1\|^2-\gamma_{K+1}(\tfrac{1}{2\tau_{K}}+\tfrac{\mu_g}{2})\|y-y_{K+1}\|^2.
		\end{aligned}
	\end{equation}
	For the second summation by assuming \eqref{eq:condition31} for $t\geq 2$, we have
	\begin{equation}\label{eq: second term}
		\begin{aligned}
			&-\sum_{t=1}^{K}\gamma_{t+1}[\langle\,A(x_{t+1}-x_t),y_{t+1}-y\rangle
			+\theta_{t-1}\langle\,A(x_t-x_{t-1}),y_{t}-y\rangle]\\
			&\leq \gamma_{t+1}\|A\|\|x_{K+1}-x_K\|\|y_{K+1}-y\|.
		\end{aligned}
	\end{equation}
	For the third summation in \eqref{eq: sum bound}, by assuming condition \eqref{eq:condition41}, we have
	\begin{equation}\label{eq: third term}
		\begin{aligned}
			-\sum_{t=1}^{K}&\gamma_{t+1}[\tfrac{1}{2\tau_t}\|y_{t+1}-y_t\|^2+(\tfrac{1}{2\eta_t}-\tfrac{L_f}{2})\|x_{t+1}-x_t\|^2
			-\theta_{t-1}\langle\,A(x_t-x_{t-1}),y_{t}-y_{t+1}\rangle]\\
			&\leq-\sum_{t=2}^{K}[\gamma_{t+1}\tfrac{1}{2\tau_t}\|y_{t+1}-y_t\|^2+\gamma_{t}(\tfrac{1}{2\eta_t}-\tfrac{L_f}{2})\|x_{t}-x_{t-1}\|^2\\
			&\quad\quad\quad-\gamma_{t+1}\theta_{t-1}\|A\|\|x_{t}-x_{t-1}\|\|y_{t+1}-y_t\|]
			-\gamma_{K+1}(\tfrac{1}{2\eta_K}-\tfrac{L_f}{2})\|x_{K+1}-x_K\|^2\\
			&\leq -\gamma_{K+1}(\tfrac{1}{2\eta_K}-\tfrac{L_f}{2})\|x_{K+1}-x_K\|^2\\
		\end{aligned}
	\end{equation}
	Therefore, from \eqref{eq: first term}, \eqref{eq: second term} and \eqref{eq: third term}, one can reestablish \eqref{eq: sum bound} as 
	\vspace{-7pt}
	\begin{equation}\label{eq: second sum bound}
		\begin{aligned}
			\sum_{t=1}^{K}\gamma_{t+1}Q(z_{t+1},z)
			&\leq \gamma_{2}(\tfrac{1}{2\eta_1}-\tfrac{\mu_f}{2})\|x-x_1\|^2-\gamma_{K+1}\tfrac{1}{2\eta_K}\|x-x_{K+1}\|^2
			+\gamma_{2}\tfrac{1}{2\tau_1}\|y-y_1\|^2\\
			&\quad-\gamma_{K+1}(\tfrac{1}{2\tau_{K}}+\tfrac{\mu_g}{2})\|y-y_{K+1}\|^2-\gamma_{K+1}(\tfrac{1}{2\eta_K}-\tfrac{L_f}{2})\|x_{K+1}-x_K\|^2\\
			&\quad+\gamma_{K+1}\|A\|\|x_{K+1}-x_K\|\|y_{K+1}-y\|.
		\end{aligned}
	\end{equation}
	\vspace{-5pt}
	Note that  $\sum_{t=1}^{K}\gamma_{t+1}Q(z_{t+1},z)$ can be rewritten since 
	\begin{equation*}
		\begin{aligned}
			-(\tfrac{1}{2\tau_{K}}+\tfrac{\mu_g}{2})\|y-y_{K+1}\|^2
			-(\tfrac{1}{2\eta_K}-\tfrac{L_f}{2})\|x_{K+1}-x_K\|^2
			&+\|A\|\|x_{K+1}-x_K\|\|y-y_{K+1}\|\\
			&\leq -\left((\tfrac{1}{\tau_{K}}+\mu_g)-\tfrac{\|A\|^2}{\tfrac{1}{\eta_{K}}-L_{{f}}}\right)\tfrac{1}{2}\|y-y_{K+1}\|^2.
		\end{aligned}
	\end{equation*}
	\vspace{-5pt}
	Note the above relation holds since
	\begin{equation*}
		\begin{aligned}
			-(\tfrac{1}{2\eta_K}-\tfrac{L_f}{2})\|x_{K+1}-x_K\|^2
			+\|A\|\|x_{K+1}-x_K\|\|y-y_{K+1}\|
			& \leq \tfrac{\|A\|^2}{\tfrac{1}{\eta_{K}}-L_{{f}}} \tfrac{1}{2}\|y-y_{K+1}\|^2.
		\end{aligned}
	\end{equation*}
	Thus 
	\vspace{-5pt}
	\begin{equation}\label{eq: final gap}
		\begin{aligned}
			\sum_{t=1}^{K}\gamma_{t+1}Q(z_{t+1},z) &\leq\tfrac{\gamma_{2}}{2} (\tfrac{1}{\eta_1} -\mu_f) \|x-x_1\|^2-\tfrac{\gamma_{K+1}}{2\eta_{K}}\|x-x_{K+1}\|^2
			+ \tfrac{\gamma_{2}}{2\tau_1}\|y-y_1\|^2\\
			&\quad-\tfrac{\gamma_{K+1}}{2}\left(\mu_g+\tfrac{1}{\tau_{K}}-\tfrac{\|A\|^2}{\tfrac{1}{\eta_{K}}-L_{{f}}}\right)\|y-y_{K+1}\|^2 .
		\end{aligned}
	\end{equation}
	Also, at $z=z^\star$, since the gap function is non-negative, we have 
	\begin{equation}\label{eq:gap at optimality }
		\begin{aligned}
			\tfrac{\gamma_{K+1}}{2}\left(\mu_g+\tfrac{1}{\tau_{K}}-\tfrac{\|A\|^2}{\tfrac{1}{\eta_{K}}-L_{{f}}}\right)\|y-y_{K+1}\|^2 
			&\leq \tfrac{\gamma_{2}}{2} (\tfrac{1}{\eta_1} -\mu_f) \|x-x_1\|^2-\tfrac{\gamma_{K+1}}{2\eta_{K}}\|x-x_{K+1}\|^2
			+ \tfrac{\gamma_{2}}{2\tau_1}\|y-y_1\|^2.
		\end{aligned}
	\end{equation}
\end{proof}
\vspace{-5pt}
As a consequence of Theorem \ref{thm: thm1} and the convexity of $\gap$ function, one can conclude the following
\begin{equation} \label{eq:GAP LPD}
	\begin{aligned}
		\gap(\bar{z}_{K+1})
		\leq\tfrac{1}{\sum_{t=1}^{K}\gamma_{t+1}} &\Big[\tfrac{\gamma_{2}}{2} (\tfrac{1}{\eta_1} -\mu_f) \|x-x_1\|^2-\tfrac{\gamma_{K+1}}{2\eta_{K}}\|x-x_{K+1}\|^2
		+ \tfrac{\gamma_{2}}{2\tau_1}\|y-y_1\|^2\\
		&\quad-\tfrac{\gamma_{K+1}}{2}\left(\mu_g+\tfrac{1}{\tau_{K}}-\tfrac{\|A\|^2}{\tfrac{1}{\eta_{K}}-L_{{f}}}\right)\|y-y_{K+1}\|^2 \Big].
	\end{aligned}
\end{equation}
Where $\bar{z}_{K+1} = \tfrac{\sum_{t=1}^K \gamma_{t+1}z_{t+1}}{\sum_{t=1}^K \gamma_{t+1}}$.
\subsection{Proof of Theorem \ref{thm: strongly-convex}}
\begin{proof}
	As one can observe, the mentioned values as step-size policy parameters satisfy the required conditions \eqref{eq:condition11}-\eqref{eq:condition41}. Additionally, from \eqref{eq: THM1.1} we know that
	\begin{equation*}
		\begin{aligned}
			\gap(\bar{z}_{K+1})\leq& \tfrac{1}{\sum_{t=1}^{K}\gamma_{t+1}}\big[\gamma_{2} \tfrac{1 }{2\eta_1}\|x-x_1\|^2+\gamma_{2}\tfrac{1}{2\tau_1}\|y-y_1\|^2\big],
		\end{aligned}
	\end{equation*}
	By considering mentioned values in \eqref{eq:strong_convex_step_policy} for the parameters, the upper bound is
	\begin{equation*}
		\begin{aligned}
			\gap(\bar{z}_{K+1})\leq& \tfrac{1}{\sum_{t=1}^{K}\tfrac{t+1}{2} + \tfrac{L_{{f}}}{\mu_f}}\big[(1+\tfrac{L_{{f}}}{\mu_f})  \tfrac{\mu_f+L_{{f}}}{2}\|x-x_1\|^2+
			(1+\tfrac{L_{{f}}}{\mu_f})\tfrac{4\|A\|^2}{2\mu_f} \|y-y_1\|^2\big].
		\end{aligned}
	\end{equation*}
	Thus
	\begin{equation*}
		\begin{aligned}
			\gap(\bar{z}_{t+1})\leq& 
			\tfrac{4}{k(k+3+\tfrac{4L_{{f}}}{\mu_f})} \big[(1+\tfrac{L_{{f}}}{\mu_f})[\tfrac{\mu_f+L_{{f}}}{2} \|x-x_1\|^2
			+\tfrac{4\|A\|^2}{2\mu_f}\|y-y_1\|^2]\big].\\
		\end{aligned}
	\end{equation*}
\end{proof}

\subsection{Proof of Theorem \ref{thm: strongly-convex dual}}
\begin{proof}
	
	First, note that the chosen values in \eqref{eq:strong_concave_step_policy} for the algorithm parameters hold the conditions \eqref{eq:condition21}-\eqref{eq:condition41} and \eqref{eq: new condition}. From the upper bound defined for the weighted gap function in \eqref{eq: sum bound} we know
	\begin{equation}
		\begin{aligned}
			\sum_{t=1}^{K}\gamma_{t+1}Q(z_{t+1},z)
			&\leq\sum_{t=1}^{K}\gamma_{t+1}  {\Large [}(\tfrac{1}{2\eta_{t}}-\tfrac{\mu_f}{2})\|x-x_t\|^2-\tfrac{1}{2\eta_t}\|x-x_{t+1}\|^2
			+\tfrac{1}{2\tau_t}\|y-y_t\|^2-(\tfrac{1}{2\tau_{t}}+\tfrac{\mu_g}{2})\|y-y_{t+1}\|^2{\Large ]} \\
			&\quad-\sum_{t=1}^{K}\gamma_{t+1}[\langle\,A(x_{t+1}-x_t),y_{t+1}-y\rangle
			+\theta_{t-1}\langle\,A(x_t-x_{t-1}),y_{t}-y\rangle]\\
			&\quad-\sum_{t=1}^{K}\gamma_{t+1}[\tfrac{1}{2\tau_t}\|y_{t+1}-y_t\|^2+(\tfrac{1}{2\eta_{t}}-\tfrac{L_f}{2})\|x_{t+1}-x_t\|^2
			-\theta_{t-1}\langle\,A(x_t-x_{t-1}),y_{t}-y_{t+1}\rangle].\\
		\end{aligned}
	\end{equation} 
	One can rewrite $\sum_{t=1}^{K}\gamma_{t+1}  \Large [(\tfrac{1}{2\eta_{t}}-\tfrac{\mu_f}{2})\|x-x_t\|^2-\tfrac{1}{2\eta_t}\|x-x_{t+1}\|^2]$ as following
	\begin{equation}
		\begin{aligned}
			=&\gamma_{2}\tfrac{\|x-x_1\|^2}{2 \eta_1} + \sum_{t=2}^{K} (\tfrac{\gamma_{t+1}}{\eta_{t}}  - \tfrac{\gamma_{t}}{\gamma_{t-1}})\tfrac{\|x-x_t\|^2}{2},\\
			\text{ From \eqref{eq: new condition} }\leq& \tfrac{\gamma_{2}}{\eta_1}  D_X^2+ (K-1)L_{{f}}D_X^2.
		\end{aligned}
	\end{equation} 
	Using the similar procedure we used in proving \eqref{eq: THM1.1}, and by the fact we showed in \eqref{eq:GAP LPD},  $\gap$ function at $\bar{z}_{K+1}$ has the following upper bound
	\begin{equation*}
		\begin{aligned}
			\gap(\bar{z}_{K+1})\leq
			&\tfrac{1}{\sum_{t=1}^{K}t+1} \big(\tfrac{\gamma_{2}}{\eta_1}  D_X^2+(K-1)L_fD_X^2+\tfrac{\gamma_{2}}{\tau_{1}}D_Y^2\big)\\
			& =\tfrac{1}{\sum_{t=1}^{K}t+1}\big( 2\tfrac{\|A\|^2}{\mu_g}D_X^2 +  2 L_{{f}}D_X^2
			+(K-1)L_{{f}}D_X^2 +\mu_g D_Y^2\big).
		\end{aligned}
	\end{equation*}
	Then
	\begin{equation*}
		\begin{aligned}
			\gap(\bar{z}_{K+1},z)&\leq \tfrac{2}{K^2}\left( \tfrac{2\|A\|^2}{\mu_g}D_X^2  +(K+1)L_{{f}}D_X^2 +\mu_g D_Y^2\right)\\
			&\quad= \tfrac{4D_X^2\|A\|^2/\mu_g  + D_Y^2\mu_g}{K^2} + \tfrac{2(K+1)L_{{f}} D_X^2}{K^2}.
		\end{aligned}
	\end{equation*}
\end{proof}
\section{General Analysis of Algorithm \ref{alg: Accelerated Lin PD for SPP linear g} (ALPD)} \label{apx:ALPD}
In this part, we focus on the proofs of the statements we mentioned in Algorithm \ref{alg: Accelerated Lin PD for SPP linear g}. Moreover, we present a new proposition (Proposition \ref{prop:prop2}) which is crucial in convergence analysis.
\subsection{Proof of Proposition  \ref{prop:prop1}}
\begin{proof}
	The approach we use here is induction. First, observe that $\beta_1=1 \in [\tfrac{1}{2}, 1]$. Now let us assume Proposition \ref{prop:prop1} is true for $\beta_{t}$ which means $\tfrac{t+1}{2}\leq \beta_t\leq t$. Let us first verify the lower bound.
	
	\textbf{Induction hypothesis  ($\beta_t\leq t$): }
	By using step-size policy for $\beta_{t+1}$ in \eqref{stepsize: case1} ($\beta_{t+1} = 1+ \theta_{t+1}\beta_{t}$), and the fact that $\theta_{t+1}\leq 1$, one can conclude that $\beta_{t+1}\leq t+1$.
	
	\textbf{Induction hypothesis  ($\beta_t\geq \tfrac{t+1}{2}$): }
	Using the similar assumptions for verifying the upper bound, we have
	\begin{equation*}
		\begin{aligned}
			\beta_{t+1} &= 1+ \theta_{t+1}\beta_{t}\\
			& \geq 1+ \tfrac{t}{t+1} \tfrac{t+1}{2}
		\end{aligned}
	\end{equation*} 
	then $\beta_{t+1} \geq 1+ \tfrac{t}{2} = \tfrac{t+2}{2}$. 
	Hence we proved that $\beta_{t+1} \in [\tfrac{t+2}{2}, t+1]$.
\end{proof}
\subsection{Statement and proof of Proposition  \ref{prop:prop2}}
Proposition \ref{prop:prop2} captures the impact of introducing $\{\beta_t\}_{t\geq 1}$ on errors incurred by linearizing $f$ in more detail. 
\begin{proposition}\label{prop:prop2}
	Let $\beta_t\geq 1$ then for all $z\in Z$, we have
	\begin{equation}\label{prop: proposition 2}			
		\begin{aligned}
			\beta_t Q(\bar{z}_{t+1},z)- (\beta_t-1)Q(\bar{z}_{t},z)
			&\leq \langle\,\nabla {f}(\underline{x}_t),x_{t+1}-x \rangle + \tfrac{L_{{f}}}{2\beta_t}\|x_{t+1}- x_t\|^2+
			\left[ g(y_{t+1})-g(y)\right] + [\phi(x_{t+1},y) - \phi(x,y_{t+1}) ].
		\end{aligned}
	\end{equation}
\end{proposition}
\begin{proof}
	From Algorithm \ref{alg: Accelerated Lin PD for SPP linear g}, one can say $\bar{x}_{t+1} - \underline{x}_t = \beta_t^{-1}\big(x_{t+1}- x_t\big)$.  Using this observation and convexity of ${f}$, we have
	\begin{equation*}
		\begin{aligned}
			\beta_t{f}(\bar{x}_{t+1})
			&\leq  \beta_t {f}(\underline{x}_t)+\beta_t\langle\,\nabla {f}(\underline{x}_t),\bar{x}_{t+1}-\underline{x}_t \rangle + \tfrac{\beta_tL_{{f}}}{2}\|\bar{x}_{t+1}- \underline{x}_t\|^2\\
			& = \beta_t {f}(\underline{x}_t) +\beta_t\langle\,\nabla {f}(\underline{x}_t),\bar{x}_{t+1}-\underline{x}_t \rangle+ \tfrac{L_{{f}}}{2\beta_t}\|x_{t+1}- x_t\|^2\\
			& = \beta_t {f}(\underline{x}_t) + (\beta_t-1)\langle\,\nabla {f}(\underline{x}_t),\bar{x}_{t}-\underline{x}_t \rangle
			+\langle\,\nabla {f}(\underline{x}_t),x_{t+1}-\underline{x}_t \rangle+\tfrac{L_{{f}}}{2\beta_t}\|x_{t+1}- x_t\|^2\\
			& = (\beta_t-1)\big[{f}(\underline{x}_t)+\langle\,\nabla {f}(\underline{x}_t),\bar{x}_{t}-\underline{x}_t \rangle\big]
			+ \big[{f}(\underline{x}_t)+\langle\,\nabla {f}(\underline{x}_t),x-\underline{x}_t \rangle\big]
			+\langle\,\nabla {f}(\underline{x}_t),x_{t+1}-x \rangle+\tfrac{L_{{f}}}{2\beta_t}\|x_{t+1}- x_t\|^2\\
			&\leq (\beta_t-1) {f}(\bar{x}_t)+ {f}(x)+\langle\,\nabla {f}(\underline{x}_t),x_{t+1}-x \rangle
			+ \tfrac{L_{{f}}}{2\beta_t}\|x_{t+1}- x_t\|^2.\\
		\end{aligned}
	\end{equation*}
	Moreover, by convexity of ${g}$ and definition of $\bar{y}_{t+1}$, we have 
	\begin{equation}
		\begin{aligned}
			\beta_t{g}(\bar{y}_{t+1}) - \beta_t{g}(y)\leq& (\beta_t-1){g}(\bar{y}_t) + {g}(y_{t+1})-\beta_t{g}(y)\\
			& = (\beta_t-1)[{g}(\bar{y}_t) -\beta_t{g}(y) ]
			+{g}(y_{t+1}) - {g}(y).
		\end{aligned}
	\end{equation}
	Also, for the coupling function, we have
	\begin{equation}\label{eq: coupling term in gap2}
		\begin{aligned}
			&\beta_t[\phi(\bar{x}_{t+1},y)-\phi(x,\bar{y}_{t+1})] - (\beta_t-1)[\phi(\bar{x}_{t},y)-\phi(x,\bar{y}_{t})]   \\
			& =[\beta_t\phi(\bar{x}_{t+1},y) - (\beta_t-1) \phi(\bar{x}_{t},y)] +[-\beta_t\phi(x,\bar{y}_{t+1})+(\beta_t-1)\phi(x,\bar{y}_{t})].
		\end{aligned}
	\end{equation}
	For the first piece in the right hand side of the above inequality, we have 
	\begin{equation*}
		\begin{aligned}
			\beta_t\phi(\bar{x}_{t+1},y) - (\beta_t-1) \phi(\bar{x}_{t},y)
			&\leq \phi(\beta_t\bar{x}_{t+1} -(\beta_t-1)\bar{x}_{t},y ) \\
			&\quad= \phi(x_{t+1},y).
		\end{aligned}
	\end{equation*}
	Note that the above inequality is based on definition of $x_{t+1}$ in Algorithm \ref{alg: Accelerated Lin PD for SPP linear g} and convexity of $\phi(\cdot,y)$ for all $y\in Y$. Similarly the second piece of \eqref{eq: coupling term in gap2} can be upper bounded as follows
	\begin{equation*}
		-\beta_t\phi(x,\bar{y}_{t+1}) + (\beta_t-1) \phi(x,\bar{y}_{t})\leq  -\phi(x,y_{t+1}),
	\end{equation*}
	From the definition of primal-dual gap function and the mentioned upper bounds for each terms, one can construct the following inequality
	\begin{equation*}
		\begin{aligned}
			\beta_t Q(\bar{z}_{t+1},z)- (\beta_t-1)Q(\bar{z}_{t},z)
			& \leq\langle\,\nabla {f}(\underline{x}_t),x_{t+1}-x \rangle +\tfrac{L_{{f}}}{2\beta_t}\|x_{t+1}- x_t\|^2
			+ \left[ {g}(y_{t+1})-{g}(y)\right]
			+ \phi(x_{t+1},y) - \phi(x,y_{t+1}).
		\end{aligned}
	\end{equation*}
\end{proof}
\subsection{Proof of Lemma \ref{lemma: lamma2}}
\begin{proof}
	Using the optimality of $y_{t+1}$ and from Lemma\ref{lemma:threepoint} for $\delta=0$, we have
	\begin{equation}\label{eq: optimality at y}
		\begin{aligned}
			\langle\,\nabla {g}(y_t),y_{t+1}-y \rangle\leq& \tfrac{1}{2\tau_{t}}\big[\|y- y_t\|^2
			-\|y- y_{t+1}\|^2-\|y_{t}- y_{t+1}\|^2\big]
			-\langle\,v_t,y-y_{t+1}\rangle.
		\end{aligned}
	\end{equation}
	Note that 
	\[
	\langle\,\nabla {g}(y_t),y_{t+1}-y \rangle =  \langle\,\nabla {g}(y_t),y_{t+1}-y_t \rangle +  \langle\,\nabla {g}(y_t),y_{t}-y \rangle.
	\]
	From strong-convexity and smoothness of ${g}$, we know that
	\[
	\langle\,\nabla {g}(y_t),y_{t+1}-y_t \rangle\geq {g}(y_{t+1})-{g}(y_t)-\tfrac{L_{{g}}}{2}\|y_t- y_{t+1}\|^2,
	\]
	and
	\[
	\langle\,\nabla {g}(y_t),y_{t}-y \rangle\geq {g}(y_{t})-{g}(y)+\tfrac{\mu_g}{2}\|y- y_{t}\|^2.
	\]
	Adding these two inequities and with \eqref{eq: optimality at y},  we can obtain an upper bound on ${g}(y_{t+1}) - {g}(y)$
	\begin{equation}\label{eq: optimal dual3}
		\begin{aligned}
			{g}(y_{t+1}) - {g}(y)&\leq \tfrac{1}{2}\big(\tfrac{1}{\tau_{t}}-\mu_g\big)\|y- y_t\|^2
			-\tfrac{1}{2}\big(\tfrac{1}{\tau_{t}}-L_{{g}}\big)\|y_t- y_{t+1}\|^2
			-\tfrac{1}{2\tau_{t}}\|y- y_{t+1}\|^2-\langle\,v_t,y-y_{t+1}\rangle.
		\end{aligned}
	\end{equation}
	Also, from the optimality of $x_{t+1}$, we have
	\begin{equation}\label{eq: optim prim3}
		\begin{aligned}
			\langle\,\nabla {f}(\underline{x}_t),x_{t+1}-x  \rangle\leq &\tfrac{1}{2\eta_t}[ \|x-x_{t}\|^2 - \|x_{t+1}-x_{t}\|^2
			-  \|x_{t+1}-x\|^2] 
			- \langle\,\nabla_x\phi(x_t,y_{t+1}),x_{t+1}-x\rangle.
		\end{aligned}
	\end{equation}
	From Proposition \ref{prop:prop2}, \eqref{eq: optimal dual3}and  \eqref{eq: optim prim3}, one can reconstruct the following upper bound for the gap function at one single iteration
	\begin{equation*}
		\begin{aligned}
			\beta_t Q(\bar{z}_{t+1},z)- (\beta_t-1)Q(\bar{z}_{t},z)
			&\leq \tfrac{1}{2\eta_t}\|x- x_t\|^2
			-\tfrac{1}{2\eta_t}\|x- x_{t+1}\|^2
			-\big(\tfrac{1}{2\eta_t} -\tfrac{L_{{f}}}{2\beta_t} \big)\|x_t- x_{t+1}\|^2
			+ \big(\tfrac{1}{2\tau_t}-\tfrac{\mu_g}{2}\big)\|y- y_t\|^2\\
			& \quad- \tfrac{1}{2\tau_t}\|y_- y_{t+1}\|^2
			-\big(\tfrac{1}{2\tau_{t}}- \tfrac{L_{{g}}}{2}\big)\|y_t- y_{t+1}\|^2
			-\langle\,v_t,y-y_{t+1} \rangle\\
			&\quad- \langle\,\nabla_x\phi(x_t,y_{t+1}),x_{t+1}-x\rangle
			+ \phi(x_{t+1},y) - \phi(x,y_{t+1}).
		\end{aligned}
	\end{equation*}
	Now, let us add and subtract $\phi(x_{t+1},y_{t+1})$ to the right hand side of above inequality, then
	\begin{equation*}
		\begin{aligned}
			\beta_t Q(\bar{z}_{t+1},z)- (\beta_t-1)Q(\bar{z}_{t},z)
			&\leq \tfrac{1}{2\eta_t}\|x- x_t\|^2
			-\tfrac{1}{2\eta_t}\|x- x_{t+1}\|^2
			-\big(\tfrac{1}{2\eta_t} -\tfrac{L_{{f}}}{2\beta_t} \big)\|x_t- x_{t+1}\|^2
			+ \big(\tfrac{1}{2\tau_t}-\tfrac{\mu_g}{2}\big)\|y- y_t\|^2\\
			& \quad- \tfrac{1}{2\tau_t}\|y_- y_{t+1}\|^2
			-\big(\tfrac{1}{2\tau_{t}}- \tfrac{L_{{g}}}{2}\big)\|y_t- y_{t+1}\|^2
			-\langle\,v_t,y-y_{t+1} \rangle\\
			&\quad- \langle\,\nabla_x\phi(x_t,y_{t+1}),x_{t+1}-x\rangle
			+ \phi(x_{t+1},y)-\phi(x_{t+1},y_{t+1})+\phi(x_{t+1},y_{t+1}) - \phi(x,y_{t+1}).
		\end{aligned}
	\end{equation*}
	By the $L_{xx}$ of $\phi(\cdot,y)$ for all $y\in Y$ of $\phi$ one can say that 
	\begin{equation} \label{eq:L_xx}
		\begin{aligned}
			&- \langle\,\nabla_x\phi(x_t,y_{t+1}),x_{t+1}-x\rangle + \phi(x_{t+1},y_{t+1}) - \phi(x,y_{t+1})\\
			&\leq \tfrac{L_{xx}}{2}\|x_t- x_{t+1}\|^2.
		\end{aligned}
	\end{equation}
	
	Based on these last two inequalities, one can immediately conclude \eqref{eq: lemma 3.1}. 
\end{proof}
\subsection{Proof of Lemma \ref{lemma:lemma3}}
\begin{proof}
	From Lemma \ref{lemma: lamma2},  and convexity of $\phi$ in $y$, we have
	\[
	\phi(x_{t+1},y) - \phi(x_{t+1},y_{t+1})\leq \langle\,\nabla_y\phi(x_{t+1},y_{t+1}),y-y_{t+1}\rangle.
	\] 
	Therefore, by the definition of $v_t$ in Algorithm \ref{alg: Accelerated Lin PD for SPP linear g}, \eqref{eq: lemma 3.1} and above inequality, we have 
	\begin{equation}
		\begin{aligned}
			\beta_t Q(\bar{z}_{t+1},z)- (\beta_t-1)Q(\bar{z}_{t},z) &\leq \tfrac{1}{2\eta_t}\|x- x_t\|^2
			-\tfrac{1}{2\eta_t}\|x- x_{t+1}\|^2 
			+ \big(\tfrac{1}{2\tau_t}-\tfrac{\mu_g}{2}\big)\|y- y_t\|^2- \tfrac{1}{2\tau_t}\|y_- y_{t+1}\|^2\\			
			&\quad-\big(\tfrac{1}{2\eta_t} -\tfrac{L_{{f}}}{2\beta_t}-\tfrac{L_{xx}}{2} \big)\|x_t- x_{t+1}\|^2 
			-\big(\tfrac{1}{2\tau_{t}}- \tfrac{L_{{g}}}{2}\big)\|y_t- y_{t+1}\|^2\\
			&\quad+\langle\,\nabla_y\phi(x_{t+1},y_{t+1})-\nabla_y\phi(x_{t},y_{t}),y-y_{t+1}\rangle\\
			&\quad-\theta_{t}\langle\,\nabla_y\phi(x_{t},y_{t})-\nabla_y\phi(x_{t-1}
			,y_{t-1}),y-y_{t+1}\rangle.\\ 
		\end{aligned}
	\end{equation}
	Notice that
	\begin{equation*}
		\begin{aligned}
			&-\theta_{t}\langle\,\nabla_y\phi(x_{t},y_{t})-\nabla_y\phi(x_{t-1},y_{t-1}),y-y_{t+1}\rangle\\ &=-\theta_{t}\langle\,\nabla_y\phi(x_{t},y_{t})-\nabla_y\phi(x_{t-1},y_{t-1}),y-y_{t}\rangle
			-\theta_{t}\langle\,\nabla_y\phi(x_{t},y_{t})-\nabla_y\phi(x_{t-1},y_{t-1}),y_t-y_{t+1}\rangle.\\
		\end{aligned}
	\end{equation*}
	Hence, the previous inequality can be written as 
	\begin{equation*}
		\begin{aligned}
			\beta_t Q(\bar{z}_{t+1},z)- (\beta_t-1)Q(\bar{z}_{t},z)  &\leq\tfrac{1}{2\eta_t}\|x- x_t\|^2-\tfrac{1}{2\eta_t}\|x- x_{t+1}\|^2
			-\big(\tfrac{1}{2\eta_t} -\tfrac{L_{{f}}}{2\beta_t}- \tfrac{L_{xx}}{2}\big)\|x_t- x_{t+1}\|^2\\
			&\quad+ \big(\tfrac{1}{2\tau_t}-\tfrac{\mu_g}{2}\big)\|y- y_t\|^2 - \tfrac{1}{2\tau_t}\|y_- y_{t+1}\|^2
			-\big(\tfrac{1}{2\tau_{t}}- \tfrac{L_{{g}}}{2}\big)\tfrac{1}{2}\|y_t- y_{t+1}\|^2\\
			&\quad+\langle\,\nabla_y\phi(x_{t+1},y_{t+1})-\nabla_y\phi(x_{t},y_{t}),y-y_{t+1}\rangle\\
			&\quad-\theta_{t}\langle\,\nabla_y\phi(x_{t},y_{t})-\nabla_y\phi(x_{t-1},y_{t-1}),y-y_{t}\rangle\\ 
			&\quad-\theta_{t}\langle\,\nabla_y\phi(x_{t},y_{t})-\nabla_y\phi(x_{t-1},y_{t-1}),y_t-y_{t+1}\rangle.
		\end{aligned}
	\end{equation*}
	Now, by multiplying both sides by $\gamma_{t}$ and letting $\theta_t = \tfrac{\gamma_{t-1}}{\gamma_{t}}, t\geq 2$, we have
	\begin{equation}\label{eq: gamma mult}
		\begin{aligned}
			\beta_t \gamma_{t}Q(\bar{z}_{t+1},z)- (\beta_t-1)\gamma_{t}Q(\bar{z}_{t},z)  &\leq \tfrac{\gamma_{t}}{2\eta_t}\|x- x_t\|^2-\tfrac{\gamma_{t}}{2\eta_t}\|x- x_{t+1}\|^2 -\gamma_{t}\big(\tfrac{1}{2\eta_t} -\tfrac{L_{{f}}}{2\beta_t}- \tfrac{L_{xx}}{2}\big)\|x_t- x_{t+1}\|^2\\
			&\quad+\gamma_{t} \big(\tfrac{1}{2\tau_t}-\tfrac{\mu_g}{2}\big)\|y- y_t\|^2 - \tfrac{\gamma_{t}}{2\tau_t}\|y_- y_{t+1}\|^2
			-\gamma_{t}\big(\tfrac{1}{2\tau_{t}}- \tfrac{L_{{g}}}{2}\big)\tfrac{1}{2}\|y_t- y_{t+1}\|^2\\
			&\quad+\gamma_{t}\langle\,\nabla_y\phi(x_{t+1},y_{t+1})-\nabla_y\phi(x_{t},y_{t}),y-y_{t+1}\rangle\\
			&\quad-\gamma_{t-1}\langle\,\nabla_y\phi(x_{t},y_{t})-\nabla_y\phi(x_{t-1},y_{t-1}),y-y_{t}\rangle\\ 
			&\quad-\gamma_{t-1}\langle\,\nabla_y\phi(x_{t},y_{t})-\nabla_y\phi(x_{t-1},y_{t-1}),y_t-y_{t+1}\rangle.
		\end{aligned}
	\end{equation}
	The last inner product can be written as follows
	\begin{equation*}
		\begin{aligned}
			&-\gamma_{t-1}\langle\,\nabla_y\phi(x_{t},y_{t})-\nabla_y\phi(x_{t-1},y_{t-1}),y_t-y_{t+1}\rangle \\ &=-\gamma_{t-1}\langle\,\nabla_y\phi(x_{t},y_{t})-\nabla_y\phi(x_{t-1},y_{t}),y_t-y_{t+1}\rangle
			-\gamma_{t-1}\langle\,\nabla_y\phi(x_{t-1},y_{t})-\nabla_y\phi(x_{t-1},y_{t-1}),y_t-y_{t+1}\rangle\\
			&\leq \gamma_{t-1} \|\nabla_y\phi(x_{t},y_{t})-\nabla_y\phi(x_{t-1},y_{t})\|\|y_t-y_{t+1}\|
			+ \gamma_{t-1} \|\nabla_y\phi(x_{t-1},y_{t})-\nabla_y\phi(x_{t-1},y_{t-1})\|\|y_t-y_{t+1}\|\\
			& \leq L_{xy} \gamma_{t-1} \|x_t-x_{t-1}\| \|y_t-y_{t+1}\| 
			+ L_{yy} \gamma_{t-1} \|y_t-y_{t-1}\| \|y_t-y_{t+1}\|.\\
		\end{aligned}
	\end{equation*}
	Since $ 0\leq \theta_{t} \leq \tfrac{\tau_{t-1}}{\tau_{t}}$ for each of norm multiplication, we have
	\begin{equation*}
		\begin{aligned}
			&L_{xy} \gamma_{t-1} \|x_t-x_{t-1}\| \|y_t-y_{t+1}\| \\&\leq\tfrac{4L_{xy}^2\gamma_{t-1}^2\tau_{t}}{2\gamma_{t}}\|x_t-x_{t-1}\|^2+ \tfrac{\gamma_{t}}{8\tau_{t}}\|y_t-y_{t+1}\|^2\\
			& \leq \tfrac{4L_{xy}^2\gamma_{t-1}\tau_{t-1}}{2}\|x_t-x_{t-1}\|^2+ \tfrac{\gamma_{t}}{8\tau_{t}}\|y_t-y_{t+1}\|^2.\\
		\end{aligned}
	\end{equation*}
	Similarly
	\begin{equation*}
		\begin{aligned}
			&L_{yy} \gamma_{t-1} \|y_t-y_{t-1}\| \|y_t-y_{t+1}\| \\ &\leq\tfrac{4L_{yy}^2\gamma_{t-1}^2\tau_{t}}{2\gamma_{t}}\|y_t-y_{t-1}\|^2+ \tfrac{\gamma_{t}}{8\tau_{t}}\|y_t-y_{t+1}\|^2\\
			&\leq \tfrac{4L_{yy}^2\gamma_{t-1}\tau_{t-1}}{2}\|y_t-y_{t-1}\|^2+ \tfrac{\gamma_{t}}{8\tau_{t}}\|y_t-y_{t+1}\|^2.\\
		\end{aligned}
	\end{equation*}
	Using these results and combining it with \eqref{eq: gamma mult} and $\beta_{t+1} -1 = \beta_t \theta_{t+1}$, we have
	\begin{equation}\label{eq:lemma2appen}
		\begin{aligned}
			\beta_t \gamma_{t}Q(\bar{z}_{t+1},z)- (\beta_t-1)\gamma_{t}Q(\bar{z}_{t},z) &\leq\tfrac{\gamma_{t}}{2\eta_t}\|x- x_t\|^2-\tfrac{\gamma_{t}}{2\eta_t}\|x- x_{t+1}\|^2+\gamma_{t} \big(\tfrac{1}{2\tau_t}-\tfrac{\mu_g}{2}\big)\|y- y_t\|^2
			-\tfrac{\gamma_{t}}{2\tau_{t}}\|y- y_{t+1}\|^2\\ 
			&\quad+\gamma_{t}\langle\,\nabla_y\phi(x_{t+1},y_{t+1})-\nabla_y\phi(x_{t},y_{t}),y-y_{t+1}\rangle\\
			&\quad-\gamma_{t-1}\langle\,\nabla_y\phi(x_{t},y_{t})-\nabla_y\phi(x_{t-1},y_{t-1}),y-y_{t}\rangle\\
			&\quad-\gamma_{t}\big(\tfrac{1}{2\eta_t} -\tfrac{L_{{f}}}{2\beta_t}-\tfrac{L_{xx}}{2}\big)\|x_t- x_{t+1}\|^2
			+\tfrac{4L_{xy}^2\gamma_{t-1}\tau_{t-1}}{2}\|x_t-x_{t-1}\|^2\\ 
			&\quad-\gamma_{t}\big(\tfrac{1}{4\tau_t} -\tfrac{L_{{g}}}{2} \big)\|y_t- y_{t+1}\|^2+\tfrac{4L_{yy}^2\gamma_{t-1}\tau_{t-1}}{2}\|y_t-y_{t-1}\|^2.
		\end{aligned}
	\end{equation}
	Applying \eqref{eq:lemma2appen} inductively and letting $x_0=x_1 , \beta_1=1$, we conclude that
	\begin{equation*}
		\begin{aligned}
			\beta_K \gamma_{K}Q(\bar{z}_{K+1},z)
			&\leq  B_K(z,z_{[K]})
			+\gamma_{K}\langle\,\nabla_y\phi(x_{K+1},y_{K+1})-\nabla_y\phi(x_{K},y_{K}),y-y_{K+1}\rangle\\
			&\quad-\gamma_{K}\big(\tfrac{1}{2\eta_K} -\tfrac{L_{{f}}}{2\beta_K} -\tfrac{L_{xx}}{2}\big)\|x_K- x_{K+1}\|^2
			-\sum_{t=1}^{K-1} \gamma_t \big(\tfrac{1}{2\eta_t}-\tfrac{L_{{f}}}{2\beta_t}-\tfrac{L_{xx}}{2}-2L_{xy}^2\tau_{t}\big)\|x_t-x_{t+1}\|^2\\
			&\quad-\gamma_{K}\big(\tfrac{1}{4\tau_K} -\tfrac{L_{{g}}}{2} \big)\|y_K- y_{K+1}\|^2 
			-\sum_{t=1}^{K-1} \gamma_t \big(\tfrac{1}{4\tau_t}-\tfrac{L_{{g}}}{2}-2L_{yy}^2\tau_{t}\big)\|y_t-y_{t+1}\|^2.
		\end{aligned}
	\end{equation*} 
	By assuming conditions in \eqref{con: stepsize condition}, one can observe that Lemma \ref{lemma:lemma3} holds.  
\end{proof}
\subsection{Proof of Theorem \ref{Thm:thm2}}
\begin{proof}
	\begin{equation*}
		\begin{aligned}
			B_K(z,z_{[K]}) 
			&= \tfrac{\gamma_{1}}{\eta_1}\tfrac{1}{2}\|x-x_{1}\|^2
			- \sum_{t=1}^{K-1}\left( \tfrac{\gamma_t}{\eta_t}-\tfrac{\gamma_{t+1}}{\eta_{t+1}}\right)\tfrac{1}{2}\|x-x_{t+1}\|^2-\tfrac{\gamma_{K}}{2\eta_{K}}\|x-x_{K+1}\|^2 \\
			&\quad+\tfrac{\gamma_{1}}{2}(\tfrac{1}{\tau_{1}}-\mu_g)\|y-y_{1}\|^2-\tfrac{\gamma_K}{2\tau_{K}}\|y-y_{K+1}\|^2
			- \sum_{t=1}^{K-1}\left(\tfrac{\gamma_{t}}{\tau_{t}} -\gamma_{t+1}(\tfrac{1}{\tau_{t+1}}-\mu_g)\right)\tfrac{1}{2}\|y-y_{t+1}\|^2\\
			&\leq\big(\tfrac{\gamma_{1}}{\eta_1} + KL_{xx}\big) D_X^2+\tfrac{\gamma_{1}}{\tau_1}D_Y^2-\tfrac{\gamma_K}{2\tau_{K}}\|y-y_{K+1}\|^2,\\
		\end{aligned}
	\end{equation*}
	where the second last inequality stems from the new condition \eqref{eq: condition 5} and the assumption that $\tfrac{\gamma_{t}}{\eta_{t}}\leq\tfrac{\gamma_{t-1}}{\eta_{t-1}} + \tfrac{L_{xx}}{2}$.
	
	Moreover,  $\gamma_{K}\langle\,\nabla_y\phi(x_{K+1},y_{K+1})-\nabla_y\phi(x_{K},y_{K}),y-y_{K+1}\rangle$ can be bounded as follows
	\begin{equation*}
		\begin{aligned}
			&\gamma_{K}\langle\,\nabla_y\phi(x_{K+1},y_{K+1})-\nabla_y\phi(x_{K},y_{K}),y-y_{K+1}\rangle	\\	&=\gamma_{K}\langle\,\nabla_y\phi(x_{K+1},y_{K+1})-\nabla_y\phi(x_{K+1},y_{K}),y-y_{K+1}\rangle  +\gamma_{K}\langle\,\nabla_y\phi(x_{t+1},y_{t})-\nabla_y\phi(x_{K},y_{K}),y-y_{K+1}\rangle\\
			&\leq\gamma_{K}L_{yy}\|y_{K}-y_{K+1}\|\|y-y_{K+1}\|
			+ \gamma_{K}L_{xy}\|x_{K}-x_{K+1}\|\|y-y_{K+1}\|\\
			&\leq \tfrac{2L_{yy}^2\gamma_{K}^2\tau_{K}}{2\gamma_{K}}\|y_K-y_{K+1}\|^2+ \tfrac{\gamma_{K}}{4\tau_{K}}\|y-y_{K+1}\|^2
			+ \tfrac{2L_{xy}^2\gamma_{K}^2\tau_{K}}{2\gamma_{K}}\|x_K-x_{K+1}\|^2+ \tfrac{\gamma_{K}}{4\tau_{K}}\|y-y_{K+1}\|^2\\
			& \leq \tfrac{2L_{yy}^2\gamma_{K}\tau_{K}}{2}\|y_K-y_{K+1}\|^2+ \tfrac{\gamma_{K}}{4\tau_{K}}\|y-y_{K+1}\|^2
			+ \tfrac{2L_{xy}^2\gamma_{K}\tau_{K}}{2}\|x_K-x_{K+1}\|^2+ \tfrac{\gamma_{K}}{4\tau_{K}}\|y-y_{K+1}\|^2.\\
		\end{aligned}
	\end{equation*}
	Then from Lemma \ref{lemma:lemma3}, we have
	\begin{equation*}\label{eq: lemma 2}
		\begin{aligned}
			\beta_K\gamma_{K}Q(\bar{z}_{K+1},z)
			&\leq B_K(z,z_{[K]})
			+\gamma_{K}\langle\,\nabla_y\phi(x_{K+1},y_{K+1})-\nabla_y\phi(x_{K},y_{K}),y-y_{K+1}\rangle\\
			&\quad-\gamma_{K}\left(  \tfrac{1}{2\eta_{K}}-\tfrac{L_{{f}}}{2\beta_K}-\tfrac{L_{xx}}{2}\right)  \|x_{K+1}-x_K\|^2
			-\gamma_{K}\left(  \tfrac{1}{4\tau_K}-\tfrac{L_{{g}}}{2}\right) \|y_K- y_{K+1}\|^2\\
			&\leq\big(\tfrac{\gamma_{1}}{\eta_1} + tL_{xx}\big) D_X^2+\tfrac{\gamma_{1}}{\tau_1}D_Y^2
			-\tfrac{\gamma_K}{2\tau_{K}}\|y-y_{K+1}\|^2+\tfrac{\gamma_{K}}{2\tau_{K}}\|y-y_{K+1}\|^2\\
			&\quad-\gamma_{K}\left(  \tfrac{1}{2\eta_{K}}-\tfrac{L_{{f}}}{2\beta_K}-\tfrac{L_{xx}}{2}-L_{xy}^2\tau_{K}\right)  \|x_{K+1}-x_K\|^2
			-\gamma_{K}\big(\tfrac{1}{4\tau_K}-\tfrac{L_{{g}}}{2}-L_{yy}^2\tau_{K}\big)\|y_K- y_{K+1}\|^2.
		\end{aligned}
	\end{equation*}
	From the conditions of Lemma \ref{lemma:lemma3}, we have
	\begin{equation}\label{eq: result of thm 1}
		\begin{aligned}
			\beta_K\gamma_{K}Q(\bar{z}_{K+1},z)
			&\leq \big(\tfrac{\gamma_{1}}{\eta_1} + KL_{xx}\big) D_X^2+\tfrac{\gamma_{1}}{\tau_1}D_Y^2.
		\end{aligned}
	\end{equation}
	Dividing both sides by $ \beta_K\gamma_{K}$ will give us \eqref{eq: Thm 1}. \end{proof}
\section{General Analysis of Algorithm \ref{alg: Accelerated Lin PD for SPP linear g prox} (Inexact ALPD) } \label{apx:inexact-ALPD}
We provide this section to highlight the similarities and important differences between ALPD and inexact ALPD algorithms in a mathematical setting. Lemma \ref{lemma:lemma4prox} shows how the dependence on $L_{xx}$ is alleviated in this approach.  
\begin{lemma}\label{lemma:lemma4prox}

	let $\bar{z}_{t+1} = (\bar{x}_{t+1}, \bar{y}_{t+1})$ and if 
	\begin{equation}\label{eq: lemma 4.1}
		\begin{aligned}
			\beta_t Q(\bar{z}_{t+1},z)- (\beta_t-1)Q(\bar{z}_{t},z)
			&\leq \tfrac{1}{2\eta_t}\|x- x_t\|^2-\tfrac{1}{2\eta_t}\|x- x_{t+1}\|^2 
			-\big(\tfrac{1}{2\eta_t} -\tfrac{L_{f}}{2\beta_t} \big)\|x_t- x_{t+1}\|^2\\
			&\quad+ \big(\tfrac{1}{2\tau_t}-\tfrac{\mu_g}{2}\big)\|y- y_t\|^2 - \tfrac{1}{2\tau_t}\|y_- y_{t+1}\|^2
			-\big(\tfrac{1}{2\tau_t} \tfrac{L_{g}}{2}\big)\tfrac{1}{2}\|y_t- y_{t+1}\|^2\\
			&\quad-\langle\,v_t,y-y_{t+1} \rangle + \phi(x_{t+1},y) - \phi(x_{t+1},y_{t+1}) + \delta_t
			+ \sqrt{2\tfrac{1}{\eta_t}\delta_t}\|x_{t+1}-x\|^2.
		\end{aligned}
	\end{equation}
	where $\delta_t$ denotes to using a $\delta_t$-approximate inexact method in primal.  
\end{lemma}
\begin{proof}
	The approach we use to prove Lemma \ref{lemma:lemma4prox} is similar to one we used in Lemma \ref{lemma: lamma2}. The only difference is rooted using the inexact method to find an $\delta_t$-approximate solution for primal which is mentioned below
	
	From  the optimality of $x_{t+1}$ using Lemma \ref{lemma:threepoint}, we have the following
	\begin{equation}\label{eq: optim prim4}
		\begin{aligned}
			\langle\,\nabla f(\underline{x}_t),x_{t+1}-x  \rangle &\leq\tfrac{1}{2\eta_t} \|x-x_{t}\|^2 -\tfrac{1}{2\eta_t} \|x_{t+1}-x_{t}\|^2
			-\tfrac{1}{2\eta_t}  \|x_{t+1}-x\|^2\\
			&\quad - \phi(x_{t+1},y_{t+1}) + \phi(x,y_{t+1})
			+ \delta_t+ \sqrt{2\tfrac{1}{\eta_t}\delta_t}\|x_{t+1}-x\|^2.
		\end{aligned}
	\end{equation} 
	Above inequality leads to the following change in \eqref{eq:L_xx} such that instead of using linear approximation of $\phi$ in $x$, we use the exact coupling function. Particularly, \eqref{eq:L_xx} changes as
	\begin{equation} 
		\begin{aligned}
			- \phi(x_{t+1},y_{t+1}) + \phi(x,y_{t+1}) + \phi(x_{t+1},y_{t+1}) - \phi(x,y_{t+1})
			= 0.
		\end{aligned}
	\end{equation}
	Observe that unlike the case in \eqref{eq:L_xx}, we do not have any dependence on $L_{xx}$. 
\end{proof}
\begin{lemma}
	
	Suppose these conditions hold
	\begin{equation}
		\begin{aligned}\label{con: stepsize condition prox}
			\beta_1 &= 1, \quad \beta_{t+1} -1 = \beta_t \theta_{t+1},
			\\0&\leq \theta_{t} \leq \tfrac{\tau_{t-1}}{\tau_{t}} \quad \tfrac{\gamma_{t}}{\eta_{t}}\leq\tfrac{\gamma_{t-1}}{\eta_{t-1}} ,\\	
			\gamma_{1}&=1,\quad \theta_{t} = \tfrac{\gamma_{t-1}}{\gamma_{t}}, \quad \tfrac{1}{2\eta_{t}}-\tfrac{L_{{f}}}{2\beta_{t}}-2L_{xy}^2\tau_{t}\geq 0,\\
			\tfrac{1}{4\tau_{t}}&-\tfrac{L_{{g}}}{2}- 2L_{yy}^2\tau_{t}\geq 0. \\
		\end{aligned}
	\end{equation}
	Then, the following inequality holds
	\begin{equation}\label{eq: lemma 4.2}
		\begin{aligned}
			\beta_K\gamma_{K}Q(\bar{z}_{K+1},z)
			& \leq B_K(z,z_{[K]})
			+\gamma_{K}\langle\,\nabla_y\phi(x_{K+1},y_{K+1})-\nabla_y\phi(x_{K},y_{K}),y-y_{K+1}\rangle\\
			&\quad+ \sum_{t=1}^K\gamma_t \sqrt{4\tfrac{1}{\eta_t}\delta_t}D_X^2-\gamma_{K}\left(  \tfrac{1}{2\eta_{K}}-\tfrac{L_{f}}{2\beta_K}\right)  \|x_{K+1}-x_K\|^2\\
			&\quad-\gamma_{K}\left(  \tfrac{1}{4\tau_K}-\tfrac{L_{g}}{2}\right) \|y_t- y_{K+1}\|^2+\sum_{t=1}^K \gamma_{t}\delta_t,\\
		\end{aligned}
	\end{equation}
	where $B_K(z,z_{[K]})$ is the following
	\begin{equation*}
		\begin{aligned}
			B_K(z,z_{[K]}) = \sum_{t=1}^K \{\tfrac{\gamma_t}{2\eta_t}[\|x-x_t\|^2-\|x-x_{t+1}\|^2] &+\gamma_t\big( \tfrac{1}{2\tau_{t}}-\tfrac{\mu_g}{2}\big)\|y-y_{t}\|^2-\tfrac{\gamma_t}{2\tau_t}\|y-y_{t+1}\|^2\}.
		\end{aligned}
	\end{equation*}
\end{lemma}
\begin{proof}
	The line of proof we follow in this lemma is the same as we used in proving Lemma \ref{lemma:lemma3}. The only difference in this case is having additional terms in the upper bound which are caused by using a $\delta_t$-approximate solution in $x$. These additional terms translate into $\sqrt{4\tfrac{1}{\eta_t}\delta_t}D_X^2$ and $\sum_{t=1}^K \gamma_{t}\delta_t$.
\end{proof}

\section{Detailed process of problem generation in Section \ref{sec:numericalex}} \label{sec:detailedfata}
\subsection{Process of problem generation in Subsection \ref{sec:ALPD_vs_LPD}}
We take the primal objective function $f(x)$ as a quadratic function of the form below 
\begin{equation}\label{eq:primal_objective}
	f(x) = \tfrac{1}{2} x^\top Qx + c^\top x ,
\end{equation}  
where $Q \in \mathbb{R}^{n\times n }$ is a positive semidefinite matrix and $c\in \mathbb{R}^n$ is a random vector with elements drawn from the standard normal distribution. We set $Q = \Lambda^\top D \Lambda$ where $\Lambda \in \mathbb{R}^{n\times n}$ is a random orthonormal matrix and $D\in \mathbb{R}_+^{n\times n}$ is a diagonal matrix whose elements are drawn from a uniform distribution between 0 and 200. To generate the orthonormal matrix $\Lambda$, first, we generate a random matrix $\bar{\Lambda} $ whose elements are drawn from the standard normal distribution. Then, we use MATLAB function orth($\bar{\Lambda}$) to return an orthonormal basis for the range of $\bar{\Lambda}$. For generating the constraint set, we sample the elements of $A\in \mathbb{R}^{m\times n}$ and $b\in \mathbb{R}^m$ from a uniform distribution between 0 and 1. In this paper, we take $n=m=100$ for each problem instance.

For the quadratic constraints, we generate randomized positive semidefinite matrices $A_j, \ j \in 	[m]$ in similar fashion as matrix $Q$. Also, $d_j, j \in [m]$ are uniformly generated in $[0,1]$. We keep $d_j$'s positive to maintain feasibility of quadratic constraints ($0$ is always feasible solution). For this case, we set $m = 10$

\subsection{Process of problem generation in Subsection \ref{sec:chambollevsours}}
The strongly-convex concave SPP is defined as below
\begin{equation}\label{eq:chamvsourss}
	\mathcal{L}(x,y):=\min_{ x\in X}\max_{ y \in Y} \{f(x) + \langle\,y,Ax-b\rangle\}. 
\end{equation} 
 Where the primal objective function $f(x)$ is defined as \eqref{eq:primal_objective} and we generate data for this problem similar to the previous section. 
 \section{Comparison of ALPD and LPD on penalty problems with different norms}\label{sec:more norms}
 In this section, we compare the performance of penalty problems where the norms are not Euclidean anymore. The instances are created similar to Section \ref{sec:detailedfata}. Figures \ref{fig:L-1 norm comparison} and \ref{fig:L-inf norm comparison} show the performances of both versions of ALPD and LPD in terms of gap function for the problem \eqref{eq:quad_lin_SPP} when $q=\infty,p=1$ and $q=1,p=\infty$ respectively. To make a better comparison, we set $L_f$ to a sufficiently large number ($L_f\approx 200$) and plot the last 50 iterates of algorithms.  Similar to the penalty problem with Euclidean norm, ALPD has a better performance.  
\begin{figure}[t]
	\centering
	\begin{minipage}{0.5\textwidth}
		\centering
		\includegraphics[width=\linewidth]{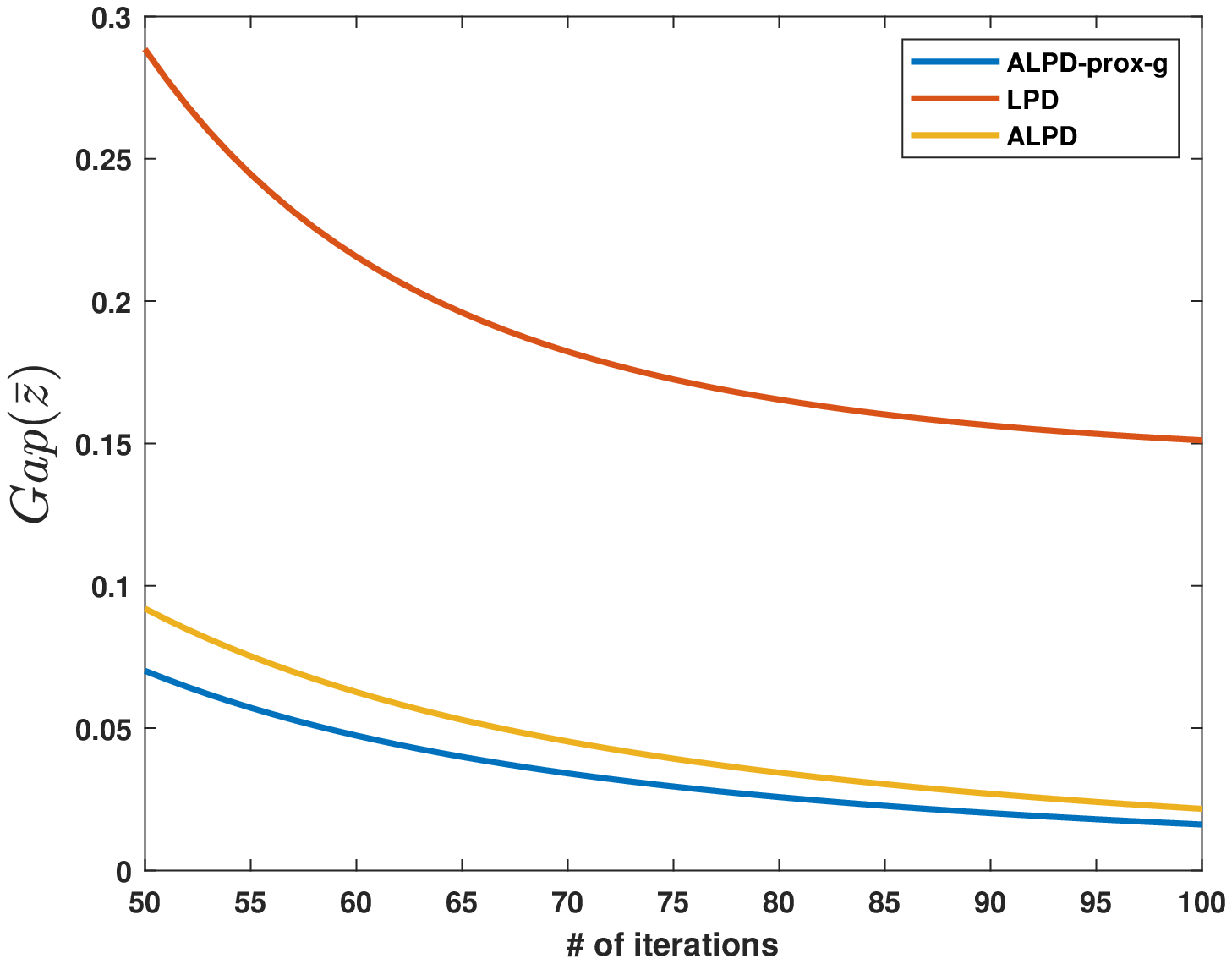}
		\caption{Comparison of the methods in terms of Gap\\ function for 10 i.i.d. replications with 100 iterations\\ in each replication for $l^\infty$-norm penalty problem.}
		\label{fig:L-1 norm comparison}
	\end{minipage}%
	\begin{minipage}{0.5\textwidth}
		\centering
		\includegraphics[width=\linewidth]{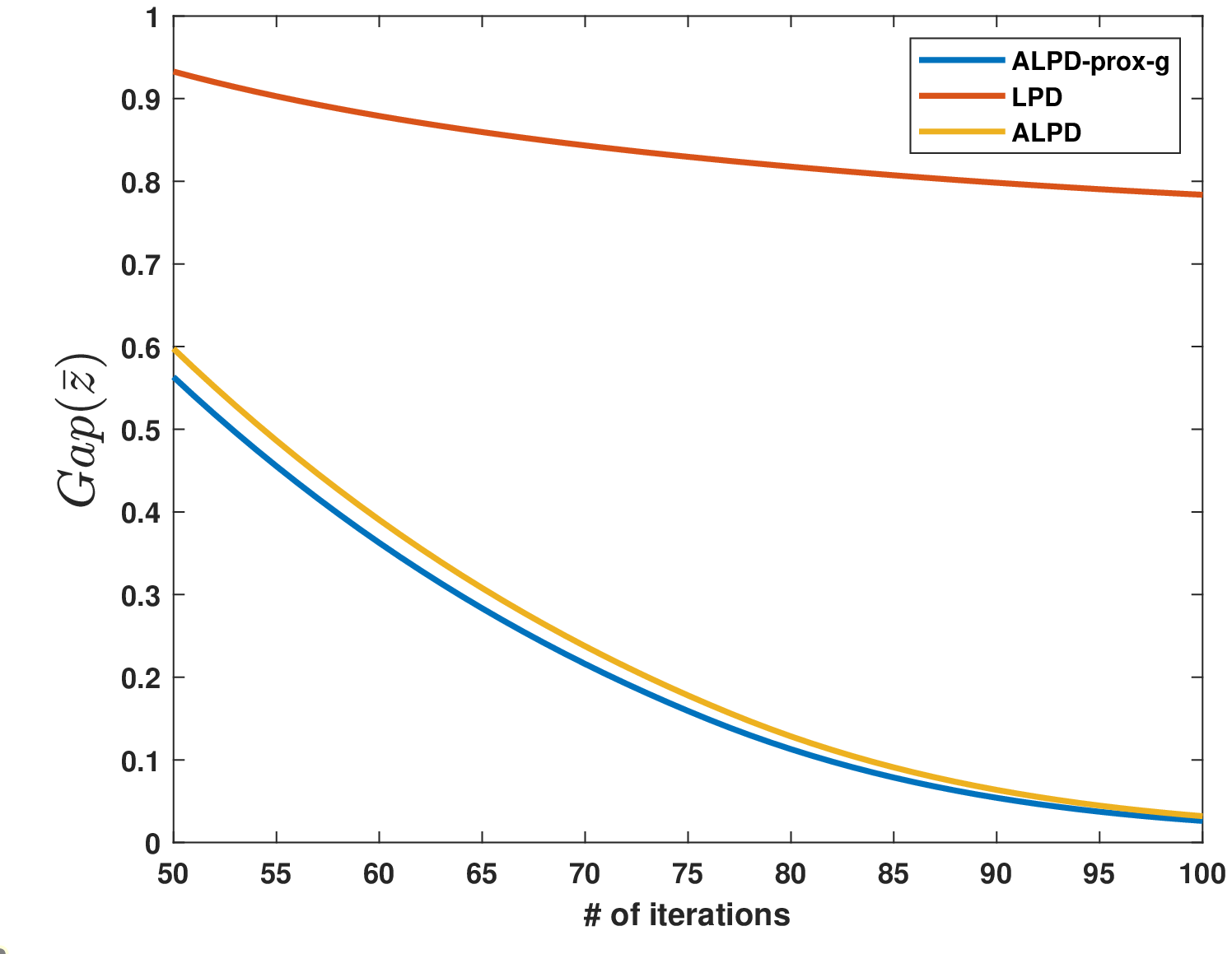}
		\caption{Comparison of the methods in terms of Gap\\ function for 10 i.i.d replications with 100 iterations\\ in each replication for $l^1$-norm penalty problem.}
		\label{fig:L-inf norm comparison}
	\end{minipage}
\end{figure}  

\section{Comparing step-size policy for two LPD algorithms}\label{apx:LPD_comparison}
Figure \ref{fig:Chambolle_Vs_ours} compares the convergence rate of the Gap function between those two step-size policies for 10 i.i.d runs with 200 iterations in each run.\ Note that the value of $L_f$ is controlled so that 200 iterations of LPD for each problem instance give a satisfactory convergence result. As one can see, our step-size policy has an advantage in terms of having faster convergence. 
\newpage
\begin{figure}[H]
	\centering
	\includegraphics[scale=0.5]{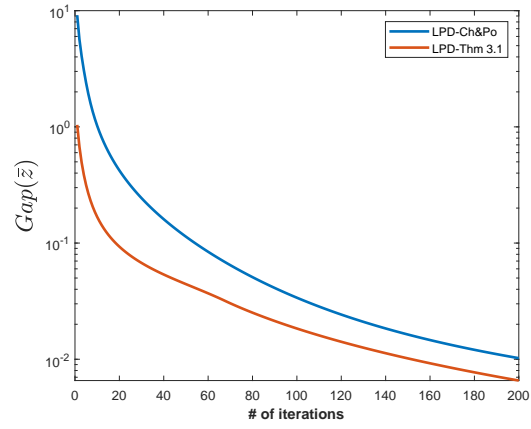}
	\caption{Comparison between the step-size policies of \eqref{eq:strong_convex_step_policy} (LPD-Thm 3.1) and \citet{chambolle2016ergodic} (LPD-Ch\&Po) for 10 i.i.d. problem instances. Both policies start from the same initial point. Note that LPD only records $\{\bar{x}_{t+1}\} _{t\geq 1}$.     }
	\label{fig:Chambolle_Vs_ours}
\end{figure}